\documentclass[a4paper,9pt]{extarticle}
\usepackage{geometry}
\usepackage{titling}
\usepackage{tabularx}
\usepackage[toc,page]{appendix}
\usepackage[colorinlistoftodos]{todonotes}
\usepackage[utf8]{inputenc}
\usepackage[english]{babel}
\usepackage{csquotes}
\usepackage{amsmath}
\usepackage{amsthm}
\usepackage{cite}
\newtheorem{theorem}{Theorem}[section]
\newtheorem{lemma}[theorem]{Lemma}
\newtheorem{corollary}[theorem]{Corollary}
\usepackage{amssymb}
\usepackage{subfigure}
\usepackage{array}
\usepackage{graphicx}
\usepackage{frcursive}
\usepackage{siunitx}
\usepackage{hyperref}
\usepackage{caption}
\usepackage{mathptmx}

\usepackage{pgfplots}
\pgfplotsset{width=10cm,compat=1.9}
\usepackage{float}

\DeclareUnicodeCharacter{05D2}{\textendash}
\DeclareUnicodeCharacter{05D0}{\textendash}
\DeclareUnicodeCharacter{05E0}{\textendash}
\DeclareUnicodeCharacter{05DC}{\textendash}
\DeclareUnicodeCharacter{05D9}{\textendash}
\DeclareUnicodeCharacter{05EA}{\textendash}

\providecommand{\keywords}[1]
{
  \small	
  \textbf{\textit{Keywords:}} #1
}

\providecommand{\akn}[1]
{
  \small	
  \textbf{\textit{Acknowledgments:}} #1
}
\title{Stability of $N$-front and $N$-back solutions in the Barkley model}
\author{Christian Kuehn\thanks{Prof. Dr.; Department of Mathematics, Technical University of Munich, Boltzmannstr. 3, 85748 Garching b. München, Germany; ckuehn@ma.tum.de} \and Pascal Sedlmeier\thanks{M.Sc.; Department of Mathematics, Technical University of Munich, Boltzmannstr. 3, 85748 Garching b. München, Germany; sedlmeier.pascal@gmail.com}} 
\date{\today}
\begin{document}
\maketitle
    \begin{abstract}
       In this paper we establish for an intermediate Reynolds number domain the stability of $N$-front and $N$-back solutions for each $N>1$ corresponding to traveling waves, in an experimentally validated model for the transition to turbulence in pipe flow proposed in \textit{[Barkley et al., Nature 526(7574):550-553, 2015]}. We base our work on the existence analysis of a heteroclinic loop between a turbulent and a laminar equilibrium proved by Engel, Kuehn and de Rijk in \textit{[Engel, Kuehn, de Rijk, Nonlinearity 35:5903, 2022]}, as well as some results from this work. The stability proof follows the verification of a set of abstract stability hypotheses stated by Sandstede in \textit{[SIAM Journal on Mathematical Analysis 29.1 (1998), pp. 183–207]} for traveling waves motivated by the FitzHugh-Nagumo equations. In particular, this completes the first detailed analysis of Engel, Kuehn and de Rijk in \textit{[Engel, Kuehn, de Rijk, Nonlinearity 35:5903, 2022]} leading to a complete existence and stability statement that nicely fits within the abstract framework of waves generated by twisted heteroclinic loops. 
    \end{abstract}
    \keywords{Barkley model, stability, $N$-front and $N$-back solutions, pipe flow, Reynolds number, traveling waves, turbulence, heteroclinic loop, reaction-diffusion-advection system.}

\pagenumbering{Roman}
\normalsize

\leftskip=0cm
\mbox{}
\pagenumbering{arabic}
\section{Introduction}
The problem of the transition from laminar to turbulent flow is an incredibly complex theoretical challenge, which is in practice extremely relevant. The underlying Navier-Stokes equations \cite{bookNVS} are highly involved with many elementary questions still open, e.g., we do not even have uniqueness of weak solutions to Navier-Stokes as proved in \cite{Non-uniquenessNVS}. Even in a simple pipe flow, the transition mechanisms from laminar to turbulent flow are not fully understood. \\

When the viscous forces dominate, they are sufficient to keep all particles of the fluid in line, corresponding to laminar flow. In contrast, when the inertial forces dominate over the viscous forces, the flow becomes turbulent (cf. \cite{Enc}). The flow is then characterized by \textcolor{black}{the} following properties:  irregularity, diffusivity, rotationality and dissipativity. \textcolor{black}{We refer to \cite{1988109}, \cite{SENGUPTA2013440} for the meaning of these terms in turbulence dynamics}. To determine whether a flow condition will be laminar or turbulent it is useful to compute the \textit{Reynolds number}, even if there does not exist a sharp boundary but a whole domain over which the transition from laminar to turbulent flow occurs. \textcolor{black}{For pipe flow, the presence of such a wide transitional range is extremely difficult to understand directly from the Navier-Stokes equations.} For example, it has been conjectured that a chaotic saddle generated by a boundary crisis \cite{Eckhardt} plays an important role. Yet, such a conjecture -- as well as other possible explanations -- \textcolor{black}{has not been mathematically rigorously proven so far} for Navier-Stokes. Recall that the Reynolds number $Re$ corresponds to the ratio of inertial forces to viscous forces and is defined for pipe flow as 
\begin{equation}
    Re:=\frac{UD}{\nu},
\end{equation}
where $U$ is the mean velocity $(m/s)$, $D$ the pipe diameter $(m)$ and $\nu$ the kinematic viscosity of the fluid $(m^2/s)$. Typically, fully turbulent flow \textcolor{black}{occurs in the region above} $Re \approx 3500$ \cite{TUMGallerkin}.\\

For intermediate Reynolds number regimes corresponding to a transitional behaviour, localized turbulence in \textcolor{black}{the} form of turbulent puffs competes with laminar background flow (cf.~\cite{Wygnanski1973OnTI}, \cite{Wygnanski1973OnTI2}). \textcolor{black}{Therefore, one possible viewpoint is to see the flow as a bistable system where turbulent and laminar flow are modelled as steady states}. The transition to fully turbulent pipe flow occurring at higher velocities is then explained by a bifurcation scenario (cf. \cite{Barkley_2015}). This viewpoint motivated the development of simpler models to capture the transition to turbulence.
 
\section{The Barkley model for pipe flow}
\label{Section 1}

To study the dynamical mechanisms and turbulence features of pipe flow, it seems natural to consider models of lower complexity than Navier-Stokes, but which still capture many interesting dynamical spatio-temporal properties. In this paper, we consider the well-established and experimentally verified model proposed by Barkley et al. (cf. \cite{Barkley_2015}):
\begin{equation}
\begin{split}
        q_t&=Dq_{xx}+(\zeta-u)q_x+f(q,u;r), \\ 
        u_t&=-uu_x+\epsilon g(q,u),
\end{split}
\label{primary}
\end{equation}
which describes pipe flow quite accurately.
The variables $q=q(x,t)$ and $u=u(x,t)$ depend only on the coordinate $x \in \mathbb{R}$ along the direction of the fluid \textcolor{black}{flow} (called the \textit{stream-wise coordinate)} and time $t\geq 0$. The variable $q$ represents the turbulence level, which can be seen as the integral of the turbulent fluctuations over a cross-section through the pipe, and $u$ the fluid centerline velocity. The functions $f=f(q,u;r)$ and $g=g(q,u)$ are the so-called \textit{reaction terms} and are given by:
\begin{equation}
    f(q,u;r)=q\big(r+u-2-(r+0.1)(q-1)\big)^2
\end{equation}
and
\begin{equation}
    g(q,u)=2-u+2q(1-u).
\end{equation}
The parameters of the dynamical system are given by $r,\,D,\,\zeta$ and $\epsilon$. The parameter $r>0$ models the Reynolds number, $D>0$ regulates the coupling of the turbulent puffs to the laminar flow (via diffusion) and $\zeta>0$ takes into consideration the fact that the time-scale of turbulent advection is slower than the time-scale of the centerline velocity. The parameter $\epsilon>0$ is taken small and controls the time-scale ratio between rapid excursions of $q$ to slow recovery of $u$ posterior to relaminarization. The Barkley model \eqref{primary} is hence nothing else than a system of reaction-diffusion equations with advective non-linearity. There seems to be a deep intrinsic connection between transition to (fully) turbulent pipe flow and traveling waves. This fact is strongly supported by experiments and numerical simulations of Navier-Stokes as described in \cite{EKR2022}. Hence, being interested in traveling waves, i.e. solutions to \eqref{primary} of the form
\begin{equation}
    \big(q(x,t),u(x,t)\big)=\big(\widetilde{q}(x-ct), \widetilde{u}(x-ct)\big)
    \label{formtilde}
\end{equation}
spreading with constant velocity $c\in\mathbb{R}$ without changing their profile, we introduce new variables $(\xi,t):=(x-ct,t)$ in which \eqref{primary} takes the form:
\begin{equation}
    \begin{split}
        q_t&=Dq_{\xi\xi}+cq_\xi+(\zeta-u)q_\xi+f(q,u;r),\\ 
    u_t&=-uu_\xi+cu_\xi+\epsilon g(q,u).
    \label{form}
    \end{split}
\end{equation}
The corresponding steady-state equation is then given by the three-dimensional system of ordinary differential equations (ODEs):
\begin{equation}
\begin{split}
\dot{q}&=s, \\
\dot{s}&=-\frac{c}{D}s-\frac{\zeta-u}{D}s-\frac{1}{D}f(q,u;r)=\frac{1}{D}\big((u+\mu)s-f(q,u;r)\big), \\
\dot{u}&=\frac{\epsilon}{u-c}g(q,u),
\label{eqdiff}
\end{split}
\end{equation}
with $\mu=-(\zeta+c)$ and the dot meaning $\frac{\partial}{\partial\xi}$. From the definition of $\zeta$, the new variable $\mu$ hence corresponds to the difference of advection between turbulence and the centerline velocity, relative to the speed $c$ of the traveling wave. The existence of a wide variety of traveling waves has been established via geometric perturbation theory in the previous work~\cite{EKR2022}, to which we refer the reader as a background. Here our main concern is stability of the traveling waves and we limit ourselves to the analysis of the Barkley model; for a broader viewpoint as well as for the origins of this paper, we refer the reader to \cite{Master_s_Thesis_Pascal_Sedlmeier}. 

\section{Proof and main results}
\label{proofbark}

To achieve the stability result, and to connect the Barkley model to other important classes of reaction-diffusion equations, we are going to prove in this section \textcolor{black}{the following: The} two theorems by Sandstede \cite[Thm.~1 and Thm.~2]{Sandstede98} for traveling waves arising in the FitzHugh-Nagumo equations also hold for traveling waves in the Barkley model \eqref{primary} for
\begin{enumerate}
    \item intermediate Reynolds number regimes $r\in(\frac{2}{3},\beta)$, where $\beta\approx 0.72946$;
    \item $\epsilon>0$ taken sufficiently small;
    \item $c<u_b(r)$, where $u_b(r)$ is the $u$-component of the equilibrium $X_2$ (cf. Section 3.1). This condition arises by requiring that the flow is directed towards the steady states $X_1$ and $X_2$ on the slow orbit segments, and is the same as in the work \cite{EKR2022} of Engel, Kuehn and de Rijk for existence of the waves.
\end{enumerate} 
In particular, with this strategy we will conclude existence and stability of $N$-fronts and $N$-backs for each $N>1$. \textcolor{black}{We will call an $N$-front a concatenation of $2N+1$ simple fronts and backs connecting between two asymptotic (steady) states; we also refer to \cite{EKR2022}, \cite{Master_s_Thesis_Pascal_Sedlmeier}, \cite{Sandstede98} for terminology and graphical representation of the profiles of various traveling waves as well as to Figure~\ref{fig: Layertime}. $N$-backs are defined analogously just with the order of the asymptotic steady states reversed.} Our proof proceeds by checking a certain number of hypotheses from Sandstede's work \cite{Sandstede98}, which are labelled there (H0)-(H7) and we follow this labeling here as well. Checking these hypotheses is then going to establish that the calculations in \cite{EKR2022} can be upgraded. \textcolor{black}{The main technical difficulties appear in Section \ref{H3} and Section \ref{H7}, where we have to (a) employ the strong $\lambda$-lemma, (b) justify the asymptotics of fronts/backs carefully, and (c) verify that certain Melnikov-type integrals are non-zero. In a broader perspective, a key difficulty was to match and upgrade certain concrete calculations carried out for existence analysis and utilize with a much broader existence/stability framework.} \textcolor{black}{In the earlier work \cite{EKR2022}, we have only shown \emph{existence} of a wide variety of traveling waves. In this current work, we are also above to cover beyond existence also \emph{stability} of $N$-front and $N$-back waves for the Barkley model.} Therefore, this current work should be viewed as a 'completion' to the existence analysis in~\cite{EKR2022} \textcolor{black}{as for any pattern-forming problem in PDEs one usually would like to have not only the existence of a pattern but also check when it is (locally asymptotically) stable.} 

\subsection{\textcolor{black}{Summary of results and definitions from \texorpdfstring{\cite{EKR2022}}{Lg} used in this paper}}
\label{summary results EKR22}
\begin{enumerate}
    \item \label{R1} \textbf{(R1)} For all values of the model Reynolds number $r$, the dynamical system \eqref{eqdiff} exhibits the equilibrium $X_1=(0,0,2)$, which corresponds to parabolic laminar flow in the pipe flow model \eqref{primary}. At this point, the turbulence level vanishes ($q=0$) and the centerline velocity takes the constant value $u=2$. \\ 
    
    In the regime $r>\frac{2}{3}$, \eqref{eqdiff} has a second equilibrium point, which is given by $X_2=(q_{b,+}(r),0,u_b(r))$ and corresponds to a turbulent steady state in the pipe flow model \eqref{primary}. The turbulence level is such that $q=q_{b,+}(r)>1$ and the centerline velocity takes the value $u=u_b(r) \in (\frac{6}{5}, \frac{4}{3})$ with: \begin{equation}
        \lim_{r \longrightarrow \frac{2}{3}}u_b(r)=\frac{4}{3}, \,\,\,\,\,\,\,\lim_{r \longrightarrow +\infty}u_b(r)=\frac{6}{5}
    \end{equation}
    \begin{equation}
        q_{b,+}(r)=1+\sqrt{\frac{r+u_b(r)-2}{r+0.1}}
    \label{eqb+}
    \end{equation}
\textcolor{black}{(from Lemma 3.1 \cite{EKR2022})}. 

\textcolor{black}{In the} \textit{slow subsystem}
\begin{equation}
\begin{split}
0 &= s, \\
  0 &= \frac{1}{D}\big((u+\mu)s-f(q,u;r)\big), \\
  (u-c)u_{\widehat{\xi}} &= g(q,u).
\end{split}
\label{slowsubsys}
\end{equation}

\textcolor{black}{with} $\widehat{\xi}=\epsilon \xi$ the stretched spatial coordinate, \textcolor{black}{the orbits are located on the nullcline
\begin{equation}
M_0:=\{(q,0,u)\in \mathbb{R}^3: f(q,u;r)=0\}
\end{equation}
called \textit{critial manifold}. We have $M_0 = M_1 \cup M_2$ with $M_1$ the line and $M_2$ the parabola:
\begin{equation}
    M_1:=\{(0,0,u) : u \in \mathbb{R}\},\,\,\,\,\,\,\, M_2:=\{(q,0,2-r+(r+0.1)(q-1)^2): q \in \mathbb{R}\},
    \label{M1M2}
\end{equation}
cf. Figure \ref{fig: Dynamicsslowsubsyst}. The equilibria of the slow subsystem \eqref{slowsubsys} and of the traveling wave equation \eqref{eqdiff} with $\epsilon>0$ are located at the intersections of $M_0$ with the second nullcline, namely the hyperbola:
\begin{equation}
    H_0:=\{(q,0,u) \in \mathbb{R}^3: g(q,u)=0\}
\end{equation}}

\begin{figure}[H]
    \centering
    \includegraphics[width=7cm]{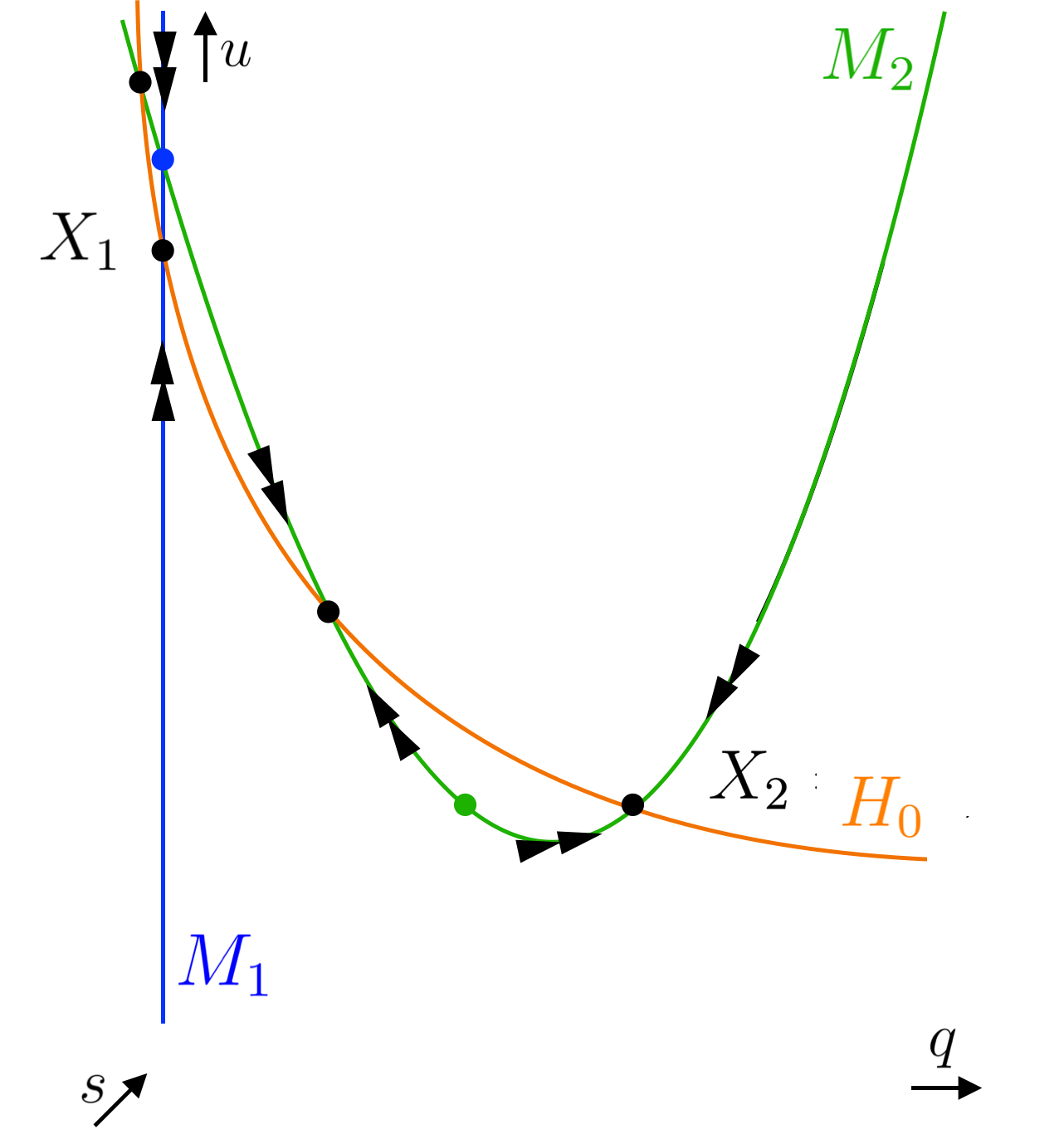}\\[3mm] 
    \caption{\textcolor{black}{Dynamics of the slow subsystem \eqref{slowsubsys} in the $(q,s,u)$-frame with $s=0$ in the regime $r>\frac{2}{3}$ and $c<u_b(r)$. Equilibria show up at the intersections of the nullclines $M_0$ and $H_0$, and $M_2$ attains its minimum at the point $(1,0,2-r)$.}}
    \label{fig: Dynamicsslowsubsyst}
\end{figure}

Setting $\epsilon=0$ in \eqref{eqdiff} yields the \textit{fast subsystem}:
\begin{equation}
    \begin{split}
    q_\xi &= s, \\
    s_\xi &= \frac{1}{D}\big((u+\mu)s-  f(q,u;r)\big), \\
    u_\xi &= 0.
    \end{split}
    \label{fastss}
    \end{equation}
    which admits:
\begin{enumerate}
    \item \textcolor{black}{In the layer $u=2$: the 2 additional equilibria $(q_{f,\pm}(r),0,2)$ located on the left and right branch of the parabola $M_2$ with:}
    \begin{equation}
        q_{f,\pm}(r)=1\pm \sqrt{\frac{r}{r+0.1}}
    \end{equation}
    \item \textcolor{black}{In the layer $u=u_b(r)$: the 2 additional equilibria $(0,0,u_b(r))$ on the line $M_1$ and $(q_{b,-}(r),0,u_b(r))$ on the left branch of the parabola $M_2$ with:}
    \begin{equation}
        q_{b,-}(r)=1-\sqrt{\frac{r+u_b(r)-2}{r+0.1}}
    \end{equation}
    \textcolor{black}{(from equations (3.7), (3.8) \cite{EKR2022})}. 
\end{enumerate}

The steady states $(q,u)=(0,2)$ and $(q,u)=(q_{b,+}(r),u_b(r))$ in \eqref{primary} are stable for $r>\frac{2}{3}$, hence \eqref{primary} is bistable for such regimes. The stability can be easily checked by writing the linearization about $X_1$ and $X_2$ of \eqref{primary} and computing the spectra.

    \item \label{L3.2} \textbf{(R2)} For $r>\frac{2}{3}$ the equilibria $X_1$ and $X_2$ in the slow subsystem \eqref{slowsubsys} are sinks if and only if $c<u_b(r)$. A sufficient condition is given by $c< \max\{2-r, \frac{6}{5} \}$. \textcolor{black}{(from Lemma 3.2 \cite{EKR2022})}

    \item \label{L3.3} \textbf{(R3)} The steady states $X_1$ and $X_2$ are hyperbolic saddles in the fast subsystem \eqref{fastss} within the layers $u=2$ (for $X_1$) and $u=u_b(r)$ (for $X_2$) in the regime $r>\frac{2}{3}$. \textcolor{black}{(from Lemma 3.3 \cite{EKR2022})}

    \item \textbf{(R4)} \textcolor{black}{We denote by $W_{\boldsymbol{\alpha}}^s(X_k)_{k=1,2}$ and $W_{\boldsymbol{\alpha}}^u(X_k)_{k=1,2}$ the \textit{stable} and \textit{unstable manifold} of the equilibrium $X_k$ for fixed parameter values $\alpha:=(D,\mu,\epsilon)$. $W_{\boldsymbol{\alpha}}^s(X_k)_{k=1,2}$ (resp. $W_{\boldsymbol{\alpha}}^u(X_k)_{k=1,2}$) is the union of all orbits in \eqref{eqdiff} which converge to $X_k$ as $\xi \longrightarrow +\infty$ (resp. $ -\infty$). Traveling waves in the pipe flow model \eqref{primary} are related to orbits in $W_{\boldsymbol{\alpha}}^s(X_1)$ and 
    $W_{\boldsymbol{\alpha}}^u(X_1)$. The profiles of such waves connect to the laminar state for $\xi \longrightarrow +\infty$ (stable manifold) and $\xi \longrightarrow -\infty$ (unstable manifold). Analogously, orbits in $W_{\boldsymbol{\alpha}}^s(X_2)$ and $W_{\boldsymbol{\alpha}}^u(X_2)$ correspond to traveling waves, with profiles connecting to the turbulent equilibrium for $\xi \longrightarrow +\infty$ and $\xi \longrightarrow -\infty$ respectively.} 

     \item \label{K10K20} \textbf{(R5)} \textcolor{black}{$K_{1,0}$ and $K_{2,0}$ are defined as the compact subsets:
    \begin{equation}
        K_{1,0}:=\{Z_1(u):u\in U_1\} \subset M_1,\,\,\,\,\,\,\,\,K_{2,0}:=\{Z_2(u):u\in U_2\}\subset M_2
    \end{equation}
    with $Z_1(u)=(0,0,u)$ and $Z_2(u)=\Big(1+\sqrt{\frac{r+u-2}{r+0.1}},0,u\Big)$. The subset $U_1$ (resp. $U_2$) of $\mathbb{R}$ is such that the orbit segments of the singular heteroclinic loop on $M_1$ (resp. the right branch of $M_2$), cf. definition \eqref{M1M2}, are strictly contained in $K_{1,0}$ (resp.  $K_{2,0}$). As described in \cite{EKR2022} we use \textit{geometric singular perturbation theory (GSPT)} to describe the sets $K_{1,0}$ and $K_{2,0}$, as well as the corresponding stable and unstable manifolds for parameter values $\alpha$ close to $\alpha_0=\alpha_0(r):=(D_0(r),\mu_0(r),0)$. One can show, following an analogous computation to the one in the proof of Lemma 3.3 (cf. \cite{EKR2022} for more details), that the manifold $K_{i,0}$ is \textit{normally hyperbolic}. GSPT then implies that $K_{i,0}$ remains as an invariant manifold $K_{i,\alpha}$ of dimension 1 in \eqref{eqdiff} depending smoothly on $\alpha$ for $\alpha$ close to $\alpha_0$.}

    \item \label{EQ 3.19} \textbf{(R6)} \textcolor{black}{We have smooth dependency on $\alpha$ for $\alpha$ close to $\alpha_0$ of the stable ($W_{\boldsymbol{\alpha}}^{s}(K_{i, \alpha})$) and unstable ($W_{\boldsymbol{\alpha}}^{u}(K_{i, \alpha})$) manifold of the invariant manifold $K_{i, \alpha}$ in \eqref{eqdiff}. At $\alpha=\alpha_0$, $W_{\boldsymbol{\alpha}}^{s}(K_{i, \alpha})$ (resp. $W_{\boldsymbol{\alpha}}^{u}(K_{i, \alpha})$) is given by the two-dimensional union of stable (resp. unstable) fibers :
\begin{equation}
    W_{\boldsymbol{\alpha_0}}^{s}(K_{i, 0})=\bigcup_{u \in U_i}W_{\boldsymbol{\alpha_0}}^{s}(Z_i(u)),\,\,\,\,\,\,\, W_{\boldsymbol{\alpha_0}}^{u}(K_{i, 0})=\bigcup_{u \in U_i}W_{\boldsymbol{\alpha_0}}^{u}(Z_i(u))
\end{equation}
with $W_{\boldsymbol{\alpha_0}}^{s}(Z_i(u))$ (resp. $W_{\boldsymbol{\alpha_0}}^{u}(Z_i(u))$) the stable (resp. unstable) manifold of dimension 1 of the equilibrium $Z_i(u) \in K_{i, 0}$ of the fast subsystem \eqref{fastss} }\textcolor{black}{(from equation (3.19) \cite{EKR2022})}

     \item \label{defQfQb} \textbf{(R7)} The approach followed in \cite{EKR2022} is to establish a heteroclinic loop in \eqref{eqdiff} through the identification of parameter values $\alpha$ close to $\alpha_0$ at which there are intersections between  $W_{\boldsymbol{\alpha}}^u(X_1)$ and $W_{\boldsymbol{\alpha}}^s(X_2)$, and between $W_{\boldsymbol{\alpha}}^s(X_1)$ and $W_{\boldsymbol{\alpha}}^u(X_2)$, using the fact that $W_{\boldsymbol{\alpha}}^s(X_k)_{k=1,2}$ coincides with  $W_{\boldsymbol{\alpha}}^s(K_{k,\alpha})$ for $\epsilon>0$. These intersections are located using Melnikov's method, cf. \cite{Holmes1980}, \cite{Melnikov63}, \cite{RemarksMelfunc}. With $\Sigma$ a plane perpendicular to the heteroclinic front $\gamma_f(\xi)$ at $\xi=0$, the unstable manifold $W_{\boldsymbol{\alpha_0}}^u(X_1)$ of dimension 1 intersects $\Sigma$ transversely at $\gamma_f(0)$. For $\alpha$ close to $\alpha_0$ there is, using smooth dependency on the parameters, a unique intersection point $X_\alpha^u$ located between $\Sigma$ and $W_{\boldsymbol{\alpha}}^u(X_1)$ such that $X^u_{\alpha_0}=\gamma_f(0)$. Using (R\ref{EQ 3.19}), the intersection of the manifold $W_{\boldsymbol{\alpha_0}}^s(K_{2,0})$ of dimension 2 and $\Sigma$ is a curve, which crosses the point $\gamma_f(0)$ and is parametrized by $u$. As a result, the vector $e_f:=(-s_f'(0;r),q_f'(0;r),0) \in \Sigma$ and the tangent vector of $\Sigma\,\,\cap W_{\boldsymbol{\alpha_0}}^s(K_{2,0})$ are transverse at $\gamma_f(0)$. Using smooth dependency on the parameters yields that $\Sigma\,\,\cap W_{\boldsymbol{\alpha}}^s(K_{2,\alpha})$ is also of dimension 1 (a curve), which depends smoothly on $\alpha$ for $\alpha$ close to $\alpha_0$. Hence the intersection of the line $l_\alpha \subset \Sigma$ through the point $X^u_\alpha$ parallel to $e_f$ with the curve $\Sigma \cap W_{\boldsymbol{\alpha}}^s(K_{2,0})$ is a unique point $X_{\alpha}^s$ for $\alpha$ close to $\alpha_0$, such that $X_{\alpha_0}^s=\gamma_f(0)$. Hence, we can write:
    \begin{equation}
        X_\alpha^s-X_\alpha^u=Q_f(\alpha;r)e_f
    \end{equation}
   with $Q_f: \mathcal{V} \times (\frac{2}{3}, +\infty) \longrightarrow \mathbb{R}$ a smooth function and $\mathcal{V} \subset \mathbb{R}^3$ a small neighborhood of $\boldsymbol{\alpha}_0$. $Q_f(\cdot;r)$, is known as a  \textit{Melnikov function}, and its roots correspond to the parameter values $\alpha$, such that the stable manifold $W^s_{\boldsymbol{\alpha}}(K_{2,\alpha})$ and the unstable manifold $W^u_{\boldsymbol{\alpha}}(X_1)$ intersect. From  
    $W^s_{\boldsymbol{\alpha}}(K_{2,\alpha})=W^s_{\boldsymbol{\alpha}}(X_2)$
    for $\epsilon>0$, we derive that such an intersection exhibits a heteroclinic front in the system \eqref{eqdiff} for $\epsilon>0$, which connects the steady states $X_1$ to $X_2$. Similarly, a Melnikov function $Q_b(\cdot;r)$ can be constructed, with the property that its roots for $\epsilon>0$ coincide with the parameter values $\alpha$, for which the stable manifold $W^s_{\boldsymbol{\alpha}}(K_{1,\alpha})=W^s_{\boldsymbol{\alpha}}(X_1)$ and the unstable manifold $W^u_{\boldsymbol{\alpha}}(X_2)$ intersect. This exhibits a heteroclinic back in \eqref{eqdiff} which connects $X_2$ to $X_1$.

    \item \label{invjacmat} \textbf{(R8)} The establishment of a heteroclinic loop in (R\ref{L3.7}) via the implicit function theorem requires the invertibility of the Jacobi matrix: \begin{equation}
    J(\boldsymbol{\alpha}_0(r);r)\textcolor{black}{:=(\partial_{D,\mu}Q_i(\alpha_0(r);r))_{i=f,b}}=\begin{pmatrix}
\frac{\partial Q_f}{\partial D}(\boldsymbol{\alpha}_0(r);r) & \frac{\partial Q_f}{\partial \mu}(\boldsymbol{\alpha}_0(r);r)\\
\frac{\partial Q_b}{\partial D}(\boldsymbol{\alpha}_0(r);r) & \frac{\partial Q_b}{\partial \mu}(\boldsymbol{\alpha}_0(r);r)
\end{pmatrix}
\end{equation}
\textcolor{black}{associated with the algebraic system of equations $Q_f(\alpha;r)=Q_b(\alpha;r)=0$.} This is the case if and only if the quantity:
\begin{equation}
    \widehat{M}(r):=\frac{{\frac{\partial Q_b}{\partial D}(\boldsymbol{\alpha}_0(r);r)}}{\frac{\partial Q_b}{\partial \mu}(\boldsymbol{\alpha}_0(r);r)}-\frac{{\frac{\partial Q_f}{\partial D}(\boldsymbol{\alpha}_0(r);r)}}{\frac{\partial Q_f}{\partial \mu}(\boldsymbol{\alpha}_0(r);r)}
    \label{Mhatde}
\end{equation}
    is non-zero. 
    
    \item \label{T2.2} \textbf{(R9)} \textcolor{black}{There are smooth functions $\mu_0:(\frac{2}{3}, +\infty) \longrightarrow (-\frac{8}{5}, \frac{1}{66}(3 \sqrt{115}-65))$ and $D_0:(\frac{2}{3}, +\infty) \longrightarrow (0, +\infty)$, which satisfy: 
\begin{equation}
        \lim_{r \longrightarrow \frac{2}{3}}\mu_0(r)=\frac{1}{66}(3 \sqrt{115}-65), \,\,\,\,\,\,\,\lim_{r \longrightarrow +\infty}\mu_0(r)=-\frac{8}{5}
    \end{equation}
and 
\begin{equation}
        \lim_{r \longrightarrow \frac{2}{3}}D_0(r)=\frac{10}{363}(34+3\sqrt{115}), \,\,\,\,\,\,\,\lim_{r \longrightarrow +\infty}D_0(r)=0,
    \end{equation}}
    such that for each fixed model Reynolds number $r>\frac{2}{3}$ satisfying $\widehat{M}(r) \neq 0$, there exists $\epsilon_0(r)>0$ such that the following holds: for each $\epsilon \in (0, \epsilon_0(r))$ there exist a diffusion rate $D=\widehat{D}(\epsilon,r)$ and a velocity $\mu=\widehat{\mu}(\epsilon,r)$ such that the dynamical system \eqref{eqdiff} admits a heteroclinic loop. It consists of a simple heteroclinic front and a simple heteroclinic back, which connect the steady states $X_1$ and $X_2$. \textcolor{black}{(from Theorem 2.2 \cite{EKR2022})}

    \item \label{melintmud} \textbf{(R10)} The \textcolor{black}{Melnikov integrals~\cite{GH}} with respect to $\mu$ and $D$ are given by 
    \begin{enumerate}
        \item along the heteroclinic front:

        \begin{equation}
        \frac{\partial Q_f}{\partial \mu}(u_f(r),D_0(r),\mu_0(r);r)=\frac{q_{f,+}(r)^2(r+0.1)}{D_0(r)^2}\widehat{M}_f(r)
        \end{equation}
    with 
\begin{equation}
\begin{split}
    \widehat{M}_f(r)&:=-\frac{q_{f,+}(r)D_0(r)\sqrt{2}(q_{f,+}(r)-2q_{f,-}(r))}{2(\mu_0(r)+2)}\int_{-\infty}^{+\infty}e^{-\sqrt{2}\big(\frac{1}{2}-\frac{q_{f,-}(r)}{q_{f,+}(r)}\big)\chi}\phi'(\chi)^2d\chi \\ &<0
    \label{Mhatf}
\end{split}
\end{equation}
 \textcolor{black}{where} $\phi: \mathbb{R} \longrightarrow \mathbb{R}$ is given by:
\begin{equation}
    \phi(\chi)=\frac{1}{1+e^{-\frac{1}{2}\sqrt{2}\chi}};
    \label{defphi}
\end{equation}
and
 \begin{equation}
        \frac{\partial Q_f}{\partial D}(u_f(r),D_0(r),\mu_0(r);r)=\frac{q_{f,+}(r)^2(r+0.1)}{D_0(r)^2}M_f(r)
        \end{equation}
        with 
    \begin{equation}
    \begin{split}
        M_f(r)&:=\frac{q_{f,+}(r)}{2}\sqrt{2}(q_{f,+}(r)-2q_{f,-}(r))\int_{-\infty}^{+\infty}e^{-\sqrt{2}\big(\frac{1}{2}-\frac{q_{f,-}(r)}{q_{f,+}(r)}\big)\chi}\phi'(\chi)^2d\chi \\ &+  \frac{0.1}{r+0.1}\int_{-\infty}^{+\infty}e^{-\sqrt{2}\big(\frac{1}{2}- \frac{q_{f,-}(r)}{q_{f,+}(r)}\big)\chi}\phi'(\chi)\phi(\chi)d\chi  \\ &- 2q_{f,+}(r)\int_{-\infty}^{+\infty}e^{-\sqrt{2}\big(\frac{1}{2}-\frac{q_{f,-}(r)}{q_{f,+}(r)}\big)\chi}\phi'(\chi)\phi(\chi)^2d\chi \\ &+ q_{f,+}(r)^2\int_{-\infty}^{+\infty}e^{-\sqrt{2}\big(\frac{1}{2}-\frac{q_{f,-}(r)}{q_{f,+}(r)}\big)\chi}\phi'(\chi)\phi(\chi)^3d\chi.
    \end{split}
    \label{Mf}
\end{equation}

\item along the heteroclinic back:

    \begin{equation}
        \frac{\partial Q_b}{\partial \mu}(u_b(r),D_0(r),\mu_0(r);r)=\frac{q_{b,+}(r)^2(r+0.1)}{D_0(r)^2}\widehat{M}_b(r)
        \end{equation}
        with \begin{equation}
\begin{split}
     \widehat{M}_b(r)&:=-\frac{q_{b,+}(r)\sqrt{2}(q_{b,+}(r)-2q_{b,-}(r))}{2(\mu_0(r)+u_b(r))}\int_{-\infty}^{+\infty}e^{-\sqrt{2}\big(\frac{1}{2}-\frac{q_{b,-}(r)}{q_{b,+}(r)}\big)\chi}\phi'(\chi)^2d\chi \\ &<0;
     \label{Mhatb}
\end{split}
\end{equation}
and 
    \begin{equation}
        \frac{\partial Q_b}{\partial D}(u_b(r),D_0(r),\mu_0(r);r)=\frac{q_{b,+}(r)^2(r+0.1)}{D_0(r)^2}M_b(r)
        \end{equation}
    with \begin{equation}
\begin{split}
     M_b(r)&:=-\frac{q_{b,+}(r)}{2}\sqrt{2}(q_{b,+}(r)-2q_{b,-}(r)) \int_{-\infty}^{+\infty}e^{-\sqrt{2}\big(\frac{1}{2}-\frac{q_{b,-}(r)}{q_{b,+}(r)}\big)\chi}\phi'(\chi)^2d\chi \\ &+\frac{u_b(r)-2.1}{r+0.1}\int_{-\infty}^{+\infty}e^{-\sqrt{2}\big(\frac{1}{2}-\frac{q_{b,-}(r)}{q_{b,+}(r)}\big)\chi}\phi'(\chi)\phi(\chi)d\chi \\ &+ 2q_{b,+}(r)\int_{-\infty}^{+\infty}e^{-\sqrt{2}\big(\frac{1}{2}-\frac{q_{b,-}(r)}{q_{b,+}(r)}\big)\chi}\phi'(\chi)\phi(\chi)^2d\chi \\ &- q_{b,+}(r)^2\int_{-\infty}^{+\infty}e^{-\sqrt{2}\big(\frac{1}{2}-\frac{q_{b,-}(r)}{q_{b,+}(r)}\big)\chi}\phi'(\chi)\phi(\chi)^3d\chi.
\end{split}
\label{Mb}
\end{equation}
\end{enumerate}
\item \label{S3.5.1} \textbf{(R11)} \textcolor{black}{The third Melnikov integral to compute (per front/ back) is with respect to $u$ and is given by:}
    \begin{enumerate}
        \item along the heteroclinic front:  \begin{equation}
        \frac{\partial Q_f}{\partial u}(u_f(r),D_0(r),\mu_0(r);r)=-\frac{1}{D_0(r)}\int_{-\infty}^{+\infty}e^{-\frac{\mu_0(r)+2}{D_0(r)}\xi}s_f(\xi;r)(s_f(\xi;r)-q_f(\xi;r))d\xi=\frac{q_{f,+}(r)^2}{D_0(r)}\widetilde{M_f}(r)
    \end{equation}
    with $\widetilde{M_f}: (\frac{2}{3},+ \infty) \longrightarrow \mathbb{R}$ defined as
    \begin{equation}
    \begin{split}
         \widetilde{M_f}(r):&=-q_{f,+}(r)\sqrt{\frac{r+0.1}{D_0(r)}}\int_{-\infty}^{+\infty}e^{-\sqrt{2}\big(\frac{1}{2}-\frac{q_{f,-}(r)}{q_{f,+}(r)}\big)\chi}\phi'(\chi)^2d\chi +\int_{-\infty}^{+\infty}e^{-\sqrt{2}\big(\frac{1}{2}-\frac{q_{f,-}(r)}{q_{f,+}(r)}\big)\chi}\phi'(\chi)\phi(\chi)d\chi,
         \label{defMftilde}
    \end{split}
    \end{equation}
  
and $q_f(\xi;r)=q_{f,+}(r)\phi\Big(q_{f,+}(r)\sqrt{\frac{r+0.1}{D}}\xi\Big)$
\textcolor{black}{(from equations (3.9), (3.11) and (3.27) \cite{EKR2022}).}
        \item along the heteroclinic back: \begin{equation}
\begin{split}
        \frac{\partial Q_b}{\partial u}(u_b(r),D_0(r),\mu_0(r);r)=-\frac{1}{D_0(r)}\int_{-\infty}^{+\infty}e^{-\frac{\mu_0(r)+u_b(r)}{D_0(r)}\xi}s_b(\xi;r)(s_b(\xi;r)-q_b(\xi;r))d\xi.
\end{split}
\end{equation}
$q_b(\xi;r)$ (resp. $s_b(\xi;r)$) denotes the $q$-component (resp. the $s$-component) of the heteroclinic back solution and $q_b(\xi;r)$ is given by (cf. \cite{EKR2022} Lemma 3.3):
\begin{equation}
    q_b(\xi;r)=q_{b,+}(r)\phi\Big(-q_{b,+}(r)\sqrt{\frac{r+0.1}{D}}\xi\Big)
\end{equation}
\textcolor{black}{(from Section 3.5.2, Lemma 3.3 \cite{EKR2022}).}
    \end{enumerate} 
 \item \label{L3.7} \textbf{(R12)} For $r > \frac{2}{3}$ and $\widehat{M}(r) \neq 0$ there exists $\epsilon_0(r)>0$ such that for all $\epsilon \in (0, \epsilon_0(r))$ there is a parameter combination $\boldsymbol{\alpha}(\epsilon;r)=(\widehat{D}(\epsilon;r),\widehat{\mu}(\epsilon;r), \epsilon)$ such that the dynamical system \eqref{eqdiff} has a heteroclinic loop. The stable $W^s_{\boldsymbol{\alpha}(\epsilon;r)}(K_{2,\boldsymbol{\alpha}(\epsilon;r)})$ and unstable $W^u_{\boldsymbol{\alpha}(\epsilon;r)}(K_{1,\boldsymbol{\alpha}(\epsilon;r)})$ manifolds intersect transversally along $\gamma_f^\epsilon$, analogously $W^s_{\boldsymbol{\alpha}(\epsilon;r)}(K_{1,\boldsymbol{\alpha}(\epsilon;r)})$ and $W^u_{\boldsymbol{\alpha}(\epsilon;r)}(K_{2,\boldsymbol{\alpha}(\epsilon;r)})$ intersect transversally along $\gamma_b^\epsilon$. Under the condition:
\begin{equation}
    \exists\,\,\beta>\frac{2}{3}\,\,\,\,\forall \,\,r\in(\frac{2}{3},\beta):\widetilde{M_f}(r)>0
\end{equation}
there exists $r_0>\frac{2}{3}$ such that the heteroclinic loop is non-degenerate for $r \in (\frac{2}{3},\beta)\cup(r_0, +\infty)$. It is double twisted for $r \in (\frac{2}{3},\beta)$.
\textcolor{black}{(from Lemma 3.7, Lemma 4.1, Lemma 4.4 \cite{EKR2022})}

\end{enumerate}

\subsection{Hypothesis \texorpdfstring{$H_0$}{Lg}: Existence of two hyperbolic equilibria \texorpdfstring{$X_1$}{Lg} and \texorpdfstring{$X_2$}{Lg}}
\label{Section2.1}

\begin{theorem}
The steady states $X_1$ and $X_2$ in \eqref{eqdiff} are hyperbolic when $\epsilon>0$ is taken sufficiently small.
\label{hyperbolicity}
\end{theorem}
\begin{proof}
The idea is to pass to the limit $\epsilon \longrightarrow 0^+$ in the dynamical system \eqref{eqdiff}, which yields the fast subsystem \eqref{fastss} and to use \textcolor{black}{(R\ref{L3.3})}. As the centerline velocity $u$ is constant, it can be seen as a system parameter. Since we consider the \textcolor{black}{regime} $r>\frac{2}{3}$, we can apply \textcolor{black}{(R\ref{L3.3})}. Hence, linearizing \eqref{eqdiff} about $X_k$ ($k\in\{1,2\}$) yields two real eigenvalues $\lambda_1^\epsilon(X_k)$ and $\lambda_3^\epsilon(X_k)$ which have opposite sign and are bounded away from $0$ as $\epsilon \longrightarrow 0^+$, which means that: 
\begin{equation}
    \exists \,\, C>0 \,\,\,\,\forall  \,\, \epsilon>0 : \vert \lambda^\epsilon_i(X_k) \vert > C
\end{equation}
for $i\in\{1,3\}$ and $k\in\{1,2\}$ and one real eigenvalue $\lambda_2^\epsilon(X_k)$ \textcolor{black}{as the slow flow is one-dimensional}. Using the stretched spatial coordinate $\widehat{\xi}=\epsilon \xi$, the dynamical system \eqref{eqdiff} transforms into:
\begin{equation}
    \begin{split}
    \epsilon q_{\widehat{\xi}} &= s, \\
    \epsilon s_{\widehat{\xi}} &= \frac{1}{D}\big((u+\mu)s-f(q,u;r)\big), \\
    (u-c)u_{\widehat{\xi}} &= g(q,u).
    \end{split}
    \label{rescaled}
\end{equation}
One readily checks that for the eigenvalues $\lambda_i^\epsilon(X_k)$ and $\widehat{\lambda_i^\epsilon}(X_k)$ of the linearizations of \eqref{eqdiff} and \eqref{rescaled} about $X_k$ the following relation holds:
\begin{equation}
    \widehat{\lambda_i^\epsilon}(X_k)=\frac{1}{\epsilon}\lambda_i^\epsilon(X_k)
    \label{rescaling relation}
\end{equation}
for $i\in\{1,2, 3\}$ and $k\in\{1,2\}$. Taking the limit $\epsilon \longrightarrow{0^+}$ in system \eqref{rescaled}, yields the slow subsystem \eqref{slowsubsys}. With $r>\frac{2}{3}$ and $c<u_b(r)$, we apply \textcolor{black}{(R\ref{L3.2})}, which implies that $X_1$ and $X_2$ are sinks in the slow subsystem \eqref{slowsubsys}. Hence, for $\epsilon>0$ sufficiently small, the linearization of \eqref{rescaled} about $(X_k)_{k=1,2}$ has an eigenvalue $\widehat{\lambda_2^\epsilon}(X_k)=\frac{1}{\epsilon}\lambda_2^\epsilon(X_k)<0$ staying bounded as $\epsilon \longrightarrow 0^+$.
 The boundedness for $\widehat{\lambda_2^\epsilon}(X_k)$ as $\epsilon \longrightarrow 0^+$  together with the rescaling relation \eqref{rescaling relation} imply that $\lambda_2^\epsilon(X_k) \underset{\epsilon \longrightarrow 0^+}{\longrightarrow}0$. 
\end{proof}
\subsection{Hypothesis \texorpdfstring{$H_1$}{Lg}: Conditions on the spectrum of the linearized vector field at \texorpdfstring{$X_1$}{Lg} and \texorpdfstring{$X_2$}{Lg}}
\label{Hypothesis1}
\textcolor{black}{\begin{corollary} The following holds:
\begin{enumerate}
    \item For the dimension of the stable $W_{\boldsymbol{\alpha}}^s(X_k)_{k=1,2}$ and unstable  $W_{\boldsymbol{\alpha}}^u(X_k)_{k=1,2}$ manifolds:
\begin{equation}
        \dim(W_{\boldsymbol{\alpha}}^s(X_k))=2,\,\,\,\,\,\,\,\,\dim(W_{\boldsymbol{\alpha}}^u(X_k))=1.
\end{equation}  
\item For the eigenvalues of the linearization of \eqref{eqdiff} about $(X_k)_{k=1,2}$ for $\epsilon>0$ sufficiently small:
\begin{enumerate}
    \item $\lambda_1^\epsilon(X_k)<\lambda_2^\epsilon(X_k)<0<\lambda_3^\epsilon(X_k)$;
    \item $-\lambda_2^\epsilon(X_k)<\lambda_3^\epsilon(X_k)$;
    \item $ \lambda_1^\epsilon(X_k)$, $\lambda_2^\epsilon(X_k)$ and $\lambda_3^\epsilon(X_k)$ are simple eigenvalues.
\end{enumerate}
\end{enumerate}
\label{Corcondspec}
\end{corollary}}

\begin{proof} \textcolor{black}{This is essentially a direct computation using the three-dimensional ODE system. For 1 \& 2 (a), (b): Follow immediately from Theorem \ref{hyperbolicity} and its proof. To 2 (c):} The Jacobi matrix associated with the linearization of \eqref{eqdiff} and evaluated at $X_k$ is a (real) $3 \times 3$ matrix, so we have $3$ eigenvalues in $\mathbb{C}$ counted with multiplicity. Since all are distinct (real), all are simple for $k\in\{1,2\}$. 
\end{proof}

\subsection{Hypothesis \texorpdfstring{$H_2$}{Lg}: Existence of two heteroclinic orbits \texorpdfstring{$\gamma_f^\epsilon(\xi)$}{Lg} and \texorpdfstring{$\gamma_b^\epsilon(\xi)$}{Lg} connecting \texorpdfstring{$X_1$}{Lg} to \texorpdfstring{$X_2$}{Lg} and vice-versa}
\label{existenceheteroclinicloop}

\textcolor{black}{In this subsection we make use of  (R\ref{L3.7}), which implies that the dynamical system \eqref{eqdiff}  has a heteroclinic orbit $\gamma_f^\epsilon$ from $X_1$ to $X_2$ (front) and a heteroclinic orbit $\gamma_b^\epsilon$ from $X_2$ to $X_1$ (back) under the condition $\widehat{M}(r)\neq 0$ in the regime $r>\frac{2}{3}$.} It turns out that the condition $\widehat{M}(r)\neq 0$ is satisfied \textcolor{black}{for all $r>\frac{2}{3}$.} Since $\widehat{M}(r)$ consists of Melnikov integrals, which can be written explicitly as a function of $r$ \textcolor{black}{(cf.  (R\ref{melintmud}))}, one could check this condition theoretically for every value of $r$. We will show it rigorously for the limit $r \longrightarrow +\infty $ \textcolor{black}{, using following lemma.
\begin{lemma}
    It holds for $r>\frac{2}{3}:$
    \begin{equation}
         \widehat{M}(r)=\frac{M_b(r)}{\widehat{M
    }_b(r)}-\frac{M_f(r)}{\widehat{M
    }_f(r)}
    \end{equation}
    \label{hatM}
\end{lemma}
\begin{proof}
    Follows directly from \textcolor{black}{(R\ref{invjacmat})} and \textcolor{black}{(R\ref{melintmud})}.
\end{proof}
This} brings us to:

\begin{theorem}
    \begin{equation}
        \exists \,\, r_0>0 \,\,\,\, \forall\,\, r>r_0: \widehat{M}(r) \neq 0.
    \end{equation}
    \label{Mhatinftynon0}
\end{theorem}
\begin{proof}
For the $u$-component of the $X_2$ equilibrium $u_b(r)$, the behaviour for \textcolor{black}{large} values of $r$ can be captured as follows, cf. \textcolor{black}{(R\ref{R1})}:
\begin{equation}
    \lim_{r \longrightarrow +\infty}u_b(r)=\frac{6}{5}.
    \label{limub}
\end{equation}
Using \textcolor{black}{(R\ref{R1})}, we obtain: 
\begin{equation}
    q_{f,\pm}(r)=1\pm\sqrt{\frac{r}{r+0.1}},
    \label{expqfpm}
\end{equation}
hence
\begin{equation}
    \lim_{r\longrightarrow +\infty} q_{f,+}(r)=2,\,\,\,\,\,\lim_{r\longrightarrow +\infty} q_{f,-}(r)=0.
    \label{qf+lim}
\end{equation}

For the layer $u=u_b(r)$ we have from \textcolor{black}{(R\ref{R1})}:
\begin{equation}
    q_{b,\pm}(r)=1\pm\sqrt{\frac{r+u_b(r)-2}{r+0.1}},
\end{equation}
hence we obtain
\begin{equation}
    \lim_{r\longrightarrow +\infty} q_{b,+}(r)=2,\,\,\,\,\,\lim_{r\longrightarrow +\infty} q_{b,-}(r)=0.
    \label{qb+lim}
\end{equation}

Using equations \eqref{limub}, \eqref{qf+lim} and \eqref{qb+lim}, we explicitly compute the limits as $r \longrightarrow +\infty$ of $M_f(r)$ and $M_b(r)$ from equations \eqref{Mf} and \eqref{Mb}, and obtain:
\begin{equation}
    \lim_{r \longrightarrow + \infty} M_f(r)=\frac{1}{3},\,\,\,\,\,\,\,\lim_{r \longrightarrow + \infty} M_b(r)=-\frac{1}{3}.
\end{equation}

With $\widehat{M}_f(r)<0$ (see equation \eqref{Mhatf}), $\widehat{M}_b(r)<0$ (see equation \eqref{Mhatb}) \textcolor{black}{and using Lemma \ref{hatM}}, we get the existence of $r_0>0$ such that:
\begin{equation}
    \widehat{M}(r)=\frac{\overbrace{M_b(r)}^{<0 }}{\underbrace{\widehat{M}_b(r)}_{<0}}-\frac{\overbrace{M_f(r)}^{>0 }}{\underbrace{\widehat{M
    }_f(r)}_{<0}}
\end{equation}
is positive for all $r>r_0$ and hence non-zero. 
\end{proof}

The condition $\widehat{M}(r)\neq 0$ is in fact satisfied for all $r>\frac{2}{3}$. This can be seen by an explicit numerical plot of the formula for the Melnikov function (cf. Figure \ref{fig: Mhatplot}), but it is cumbersome to verify the involved explicit formulas analytically for all $r>\frac{2}{3}$. Since this is certainly possible using interval arithmetic techniques for integrals, we just demonstrated here the case for large $r$. 
Hence, the solutions $\gamma_f^\epsilon$ and $\gamma_b^\epsilon$ fulfill \textcolor{black}{in the regime $r>\frac{2}{3}$}:
\begin{equation}
    \lim_{\xi \longrightarrow -\infty}\gamma_f^\epsilon(\xi)=X_1,\,\,\ \lim_{\xi \longrightarrow +\infty}\gamma_f^\epsilon(\xi)=X_2
    \label{limit1}
\end{equation}
and 
\begin{equation}
    \lim_{\xi \longrightarrow -\infty}\gamma_b^\epsilon(\xi)=X_2,\,\,\ \lim_{\xi \longrightarrow +\infty}\gamma_b^\epsilon(\xi)=X_1.
    \label{limit2}
\end{equation}

\begin{figure}[H]
    \centering
    \includegraphics[width=11cm]{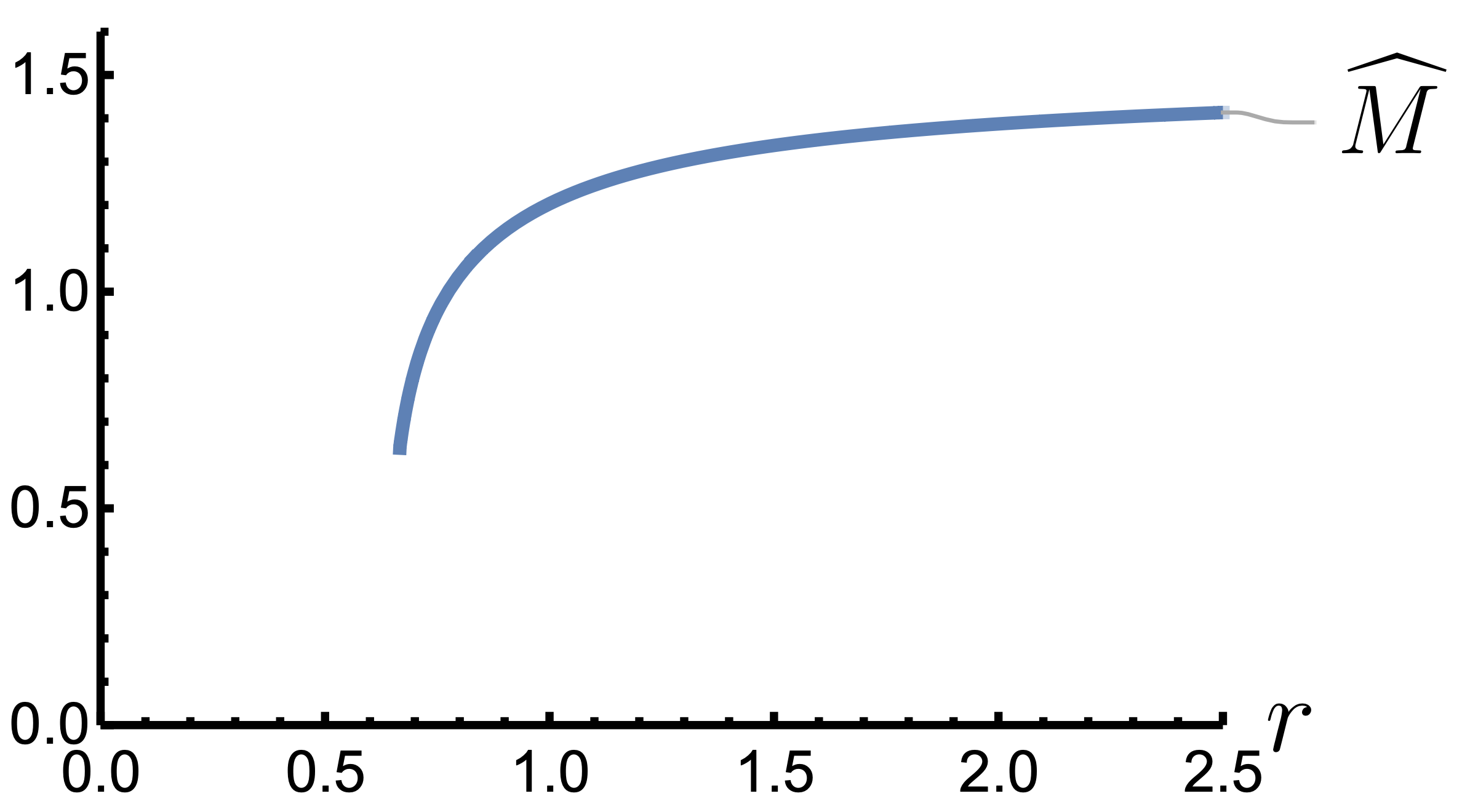}\\[3mm] 
    \caption{Representation of $\widehat{M}$ as a function of the model Reynolds number $r$ for $r\in(\frac{2}{3},\frac{5}{2})$ from \cite{EKR2022}. \textcolor{black}{Note that the grey curve just indicates that $\widehat{M}$ is given by the blue curve.} }
    \label{fig: Mhatplot}
\end{figure}

\subsection{Hypothesis \texorpdfstring{$H_3$}{Lg}: Non-degeneracy of the heteroclinic solutions}
\label{H3}

\begin{theorem}
There exists \textcolor{black}{$\beta$} $>\frac{2}{3}$, such that for  $I=(\frac{2}{3},\beta)$ the function $ \widetilde{M_f}$ \textcolor{black}{defined in \eqref{defMftilde} }satisfies:
\begin{equation}
    \forall\,\, r \in I : \,\, \widetilde{M_f}(r)>0.
\end{equation}
\label{thmint}
\end{theorem}
\begin{proof}
   From \textcolor{black}{(R\ref{S3.5.1})}, $\frac{\partial Q_f}{\partial u}$ and  $\widetilde{M_f}$ are related by:
    \begin{equation}
        \frac{\partial Q_f}{\partial u}(u_f(r),D_0(r),\mu_0(r);r)=\frac{q_{f,+}(r)^2}{D_0(r)}\widetilde{M_f}(r).
    \end{equation}
    
Now let us compute explicitly the limit $r \longrightarrow \frac{2}{3}^+$ in $\widetilde{M_f}(r)$. With equation \eqref{expqfpm} we obtain:
\begin{equation}
    \lim_{r \longrightarrow \frac{2}{3}^+} q_{f,\pm}(r)=1 \pm \sqrt{\frac{20}{23}}.
\end{equation}
The existence result of smooth functions $\mu_0$ and $D_0$ \textcolor{black}{(R\ref{T2.2})} establishes the following property for $D_0$ as $r \longrightarrow \frac{2}{3}^+$:
\begin{equation}
    \lim_{r \longrightarrow \frac{2}{3}^+}D_0(r)=\frac{10}{363}(34+3\sqrt{115}).
\end{equation}
Let
\begin{equation}
    a:=\frac{1}{2}-\frac{1-\sqrt{\frac{20}{23}}}{1+\sqrt{\frac{20}{23}}},\,\,\,\,\,\,b:=\sqrt{\frac{\frac{2}{3}+\frac{1}{10}}{\frac{10}{363}(34+3\sqrt{115})}}=\frac{1}{10}\sqrt{\frac{2783}{ 34+3\sqrt{115}}},\,\,\,\,\,\,c:=-\Big(1+\sqrt{\frac{20}{23}}\Big).
\end{equation}

Hence, we have:
\begin{equation}
    \lim_{r \longrightarrow \frac{2}{3}^+}\widetilde{M_f}(r)=bc\int_{-\infty}^{+\infty}e^{-\sqrt{2}a\chi}\phi'(\chi)^2d\chi+\int_{-\infty}^{+\infty}e^{-\sqrt{2}a\chi}\phi'(\chi)\phi(\chi)d\chi.
\end{equation}
Together with the definition \eqref{defphi} of $\phi$ and the computation of its derivative, we get:
\begin{equation}
    \lim_{r \longrightarrow \frac{2}{3}^+}\widetilde{M_f}(r)=\frac{bc}{2}\int_{-\infty}^{+\infty}\frac{e^{-\sqrt{2}(a+1)\chi}}{\big(1+e^{-\frac{\sqrt{2}}{2}\chi}\big)^4}d\chi+\frac{\sqrt{2}}{2}\int_{-\infty}^{+\infty}\frac{e^{-\sqrt{2}(a+\frac{1}{2})\chi}}{\big(1+e^{-\frac{\sqrt{2}}{2}\chi}\big)^3}d\chi.
\end{equation}
With
\begin{equation}
    \int_{-\infty}^{+\infty}\frac{e^{-\sqrt{2}(a+1)\chi}}{\big(1+e^{-\frac{\sqrt{2}}{2}\chi}\big)^4}d\chi \approx 0.426,\,\,\,\,\,\,\int_{-\infty}^{+\infty}\frac{e^{-\sqrt{2}(a+\frac{1}{2})\chi}}{\big(1+e^{-\frac{\sqrt{2}}{2}\chi}\big)^3}d\chi \approx 0.663
\end{equation}

both computed with MATLAB (\cite{MATLAB}), it holds for the limit of $\widetilde{M_f}(r)$ as $r \longrightarrow \frac{2}{3}^+$:
\begin{equation}
    \lim_{r \longrightarrow \frac{2}{3}^+} \widetilde{M_f}(r) \approx 0.202>0
\end{equation}

Due to the continuity of $\widetilde{M_f}$ -- which is clear from its definition \eqref{defMftilde} -- such an open interval $I$ with lower bound $\frac{2}{3}$ indeed exists. 
\end{proof}

\textcolor{black}{As before, we remark that the rigorous numerical computation of an explicit sign of an integral would be required to get a large interval, i.e., to increase $\beta$. This is a relatively easy task using interval arithmetic and asymptotics, which we have omitted here as it does not contribute to the main line of the argument. In the proof of the next theorem, we need the explicit computation of two anti-derivatives, } \textcolor{black}{which is an essential step to evaluate these new upcoming integrals over $\mathbb{R}$.}
\textcolor{black}{\begin{lemma}
    It holds on each integration interval $J$ subset of $\mathbb{R}$: 
       \begin{equation}
        \int\frac{e^{-\sqrt{2}\chi}}{\big(1+e^{-\frac{\sqrt{2}}{2}\chi}\big)^3}d\chi=-\frac{1}{\sqrt{2}\big(e^{\frac{\sqrt{2}}{2}\chi}+1\big)^2},\,\,\,\,\,\,\int\frac{e^{-\frac{3\sqrt{2}}{2}\chi}}{\big(1+e^{-\frac{\sqrt{2}}{2}\chi}\big)^4}d\chi=-\frac{\sqrt{2}}{3\big(e^{\frac{\sqrt{2}}{2}\chi}+1\big)^3}
    \end{equation}
    \label{lemma anti-derivatives}
\end{lemma}}
\begin{proof}
    Immediate by computation.
\end{proof}
\begin{theorem}
     $\widetilde{M_f}$ has the following behaviour as $r \longrightarrow + \infty$:
     \begin{equation}
    \lim_{r \longrightarrow +\infty} \widetilde{M_f}(r) = -\infty.
\end{equation}
\label{assbehMftilde}
\end{theorem}
\begin{proof}
    To prove this result, we use expression \eqref{defMftilde} and \eqref{qf+lim}, as well as following asymptotic behaviour of $D_0$ from \textcolor{black}{(R\ref{T2.2})}:
    \begin{equation}
        \lim_{r \longrightarrow + \infty}D_0(r)=0^+.
        \label{limD0}
    \end{equation}
    The second term of the sum \eqref{defMftilde} in the asympotic limit $r\longrightarrow + \infty$ is the integral
    \begin{equation}
        \int_{-\infty}^{+\infty}e^{-\frac{\sqrt{2}}{2}\chi}\phi'(\chi)\phi(\chi)d\chi= \frac{\sqrt{2}}{2}\int_{-\infty}^{+\infty}\frac{e^{-\sqrt{2}\chi}}{\big(1+e^{-\frac{\sqrt{2}}{2}\chi}\big)^3}d\chi.
    \end{equation}
     \textcolor{black}{Using Lemma \ref{lemma anti-derivatives} yields immediately:}
    \begin{equation}
        \frac{\sqrt{2}}{2}\int_{-\infty}^{+\infty}\frac{e^{-\sqrt{2}\chi}}{\big(1+e^{-\frac{\sqrt{2}}{2}\chi}\big)^3}d\chi=\frac{1}{2}.
    \end{equation}
    This shows in particular that the second term in the sum \eqref{defMftilde} is bounded as $r \longrightarrow + \infty$. Also the integral appearing in the first term of \eqref{defMftilde}:
    \begin{equation}
        \int_{-\infty}^{+\infty}e^{-\frac{\sqrt{2}}{2}\chi}\phi'(\chi)^2d\chi=\frac{1}{2}\int_{-\infty}^{+\infty}\frac{e^{-\frac{3\sqrt{2}}{2}\chi}}{\big(1+e^{-\frac{\sqrt{2}}{2}\chi}\big)^4}d\chi
        \label{positivityint}
    \end{equation}
    \textcolor{black}{can be explicitly computed using Lemma \ref{lemma anti-derivatives}, as:} 
    \begin{equation}
        \frac{1}{2}\int_{-\infty}^{+\infty}\frac{e^{-\frac{3\sqrt{2}}{2}\chi}}{\big(1+e^{-\frac{\sqrt{2}}{2}\chi}\big)^4}d\chi=\frac{1}{3\sqrt{2}},
    \end{equation}
    so the integral is bounded and positive (the positivity could be seen directly from expression \eqref{positivityint}). Using these two results for the integrals appearing in each term of expression  \eqref{defMftilde} for $\widetilde{M_f}(r)$, as well as \eqref{qf+lim} and \eqref{limD0} yields:
    \begin{equation}
        \lim_{r \longrightarrow +\infty} \widetilde{M_f}(r) = -\infty
    \end{equation}
    as required, finishing the proof.
\end{proof}
Hence, the interval $I$ in Theorem \ref{thmint} is a strict subset of $(\frac{2}{3},+\infty)$. Hence the positivity condition for $\widetilde{M_f}$ does not hold for every value of $r$. The numerical computation in \cite{EKR2022} shows that 
\begin{equation}
    \beta\approx 0.72946
    \label{beta}
\end{equation} can be taken, which means:
\begin{equation}
    \forall\,\,r\in\,\,I=(\frac{2}{3},\beta):\widetilde{M_f}(r)>0.
\end{equation}
\textcolor{black}{(R\ref{L3.7})} states the non-degeneracy of the heteroclinic loop established in Section \ref{existenceheteroclinicloop} \textcolor{black}{if there exists $\beta>\frac{2}{3}$ such that for all $r \in (\frac{2}{3}, \beta)$ the function $\widetilde{M_f}: (\frac{2}{3},+ \infty) \longrightarrow \mathbb{R}$ satisfies $\widetilde{M_f}(r)>0$,} which is verified theoretically in Theorem \ref{thmint}. The optimal (in the sense of largest possible) value is chosen \textcolor{black}{as in equation \eqref{beta}} in accordance with the numerical computation plotted in \cite{EKR2022} Figure 11b. Note that $\widetilde{M_f}$ is strictly decreasing suggesting that for $r>\frac{2}{3}$ the positivity condition is satisfied \textit{if and only if} $r\in(\frac{2}{3},\beta)$. \textcolor{black}{(R\ref{L3.7})} establishes the existence of $r_0>\frac{2}{3}$ such that for all $r\in(\frac{2}{3},\beta) \cup (r_0, +\infty)$ the heteroclinic loop consisting of the heteroclinic front $\gamma^\epsilon_f(\xi)$ and the heteroclinic back $\gamma^\epsilon_b(\xi)$ is non-degenerate. This means that:
\begin{enumerate}
    \item \begin{enumerate}
        \item \underline{For $\xi \longrightarrow +\infty$ :} $\gamma^\epsilon_f(\xi)$ is asymptotically tangent to the principal \textbf{stable} eigenvector $e_2(X_2)$ of the steady state $X_2=(q_{b,+}(r),0,u_b(r))$. Recall $\lambda_2(X_k)$ is the principal stable eigenvalue with corresponding eigenvector $e_2(X_k)$ for $k\in \{1,2\}$, cf. Section \ref{Hypothesis1}.
        \item \underline{For $\xi \longrightarrow -\infty$ :} $\gamma^\epsilon_f(\xi)$ is asymptotically tangent to the principal \textbf{unstable} eigenvector $e_3(X_1)$ of the steady state $X_1=(0,0,2)$. Recall $\lambda_3(X_k)$ is the principal unstable eigenvalue with corresponding eigenvector $e_3(X_k)$ for $k\in \{1,2\}$, cf. Section \ref{Hypothesis1}.
        \item \underline{For $\xi \longrightarrow +\infty$ :} $\gamma^\epsilon_b(\xi)$ is asymptotically tangent to the principal \textbf{stable} eigenvector $e_2(X_1)$ of the steady state $X_1$.
        \item \underline{For $\xi \longrightarrow -\infty$ :} $\gamma^\epsilon_b(\xi)$ is asymptotically tangent to the principal \textbf{unstable} eigenvector $e_3(X_2)$ of the steady state $X_2$.
    \end{enumerate}
    \item The strong inclination conditions hold:
    \begin{equation}
        \lim_{\xi \longrightarrow -\infty}T_{\gamma_f^\epsilon(\xi)}W^s_{\boldsymbol{\alpha}(\epsilon;r)}(X_2)=T_{X_1}W^u_{\boldsymbol{\alpha}(\epsilon;r)}(X_1)+T_{X_1}W^{ss}_{\boldsymbol{\alpha}(\epsilon;r)}(X_1)
        \label{cond1}
    \end{equation}
    and 
    \begin{equation}
        \lim_{\xi \longrightarrow -\infty}T_{\gamma_b^\epsilon(\xi)}W^s_{\boldsymbol{\alpha}(\epsilon;r)}(X_1)=T_{X_2}W^u_{\boldsymbol{\alpha}(\epsilon;r)}(X_2)+T_{X_2}W^{ss}_{\boldsymbol{\alpha}(\epsilon;r)}(X_2),
        \label{cond2}
    \end{equation}
    where $T_xW$ means the tangent space of a given manifold $W$ at the base point $x \in W$. $W^s_{\boldsymbol{\alpha}(\epsilon;r)}(X_k)$, $W^u_{\boldsymbol{\alpha}(\epsilon;r)}(X_k)$ and $W^{ss}_{\boldsymbol{\alpha}(\epsilon;r)}(X_k)$ denote the stable, unstable and strong stable manifolds of $X_k$ respectively. The dimensions of these manifolds are given by for $k\in \{1,2\}$:
    \begin{equation}
        \dim (W^s_{\boldsymbol{\alpha}(\epsilon;r)}(X_k))=2,\,\,\,\,\,\,\,\dim (W^u_{\boldsymbol{\alpha}(\epsilon;r)}(X_k))=1,\,\,\,\,\,\,\,\dim (W^{ss}_{\boldsymbol{\alpha}(\epsilon;r)}(X_k))=1.
    \end{equation}
\end{enumerate}
We justify latter two points as follows:
\begin{enumerate}
    \item 
    \begin{enumerate}
        \item First of all, we observe that
\begin{equation}
    \gamma^\epsilon_f \nsubseteq W^{ss}_{\boldsymbol{\alpha}(\epsilon;r)}(X_2).
    \label{nonincl}
\end{equation}
\textcolor{black}{This result, which follows by continuity, is a central element to verify the asymptotics of the heteroclinic front as $\xi \longrightarrow +\infty$, key here is to consider the $u$-component.} In the fast subsystem \eqref{fastss} the strong manifold $W^{ss}_{\boldsymbol{\alpha_0}(r)}(X_2)$ lies in the layer $u=u_b(r)$, where $\alpha_0=(D_0(r),\mu_0(r),0)$ denotes the set of parameter values at which the singular heteroclinic loop exists. However, in the limit $\epsilon \longrightarrow 0^+$, the heteroclinic front $\gamma_f^\epsilon$ converges to the connection $X_1 \longrightarrow X_2$ of the singular heteroclinic front (cf. Figure \ref{fig: Singloop}).  This connection is made up of two parts: the heteroclinic connection $X_f$ in the fast subsystem \eqref{fastss} -- which lies in the layer $u=2$ -- and the slow orbit segment in the slow subsystem \eqref{slowsubsys} -- which lies on the manifold $M_0$ -- connecting $(q_{f,+}(r), 0, 2)$ to $X_2$. As corollary of equation \eqref{nonincl}, for $\xi \longrightarrow +\infty$, the heteroclinic front $\gamma_f^\epsilon(\xi)$ is asymptotically tangent to the principal stable eigenvector $e_2(X_2)$ of $X_2$. 
\item The convergence tangent to the principal unstable eigenvector $e_3(X_1)$ of $X_1$ as $\xi \longrightarrow -\infty$ is straightforward, since the spectral decomposition of the linearized vector field at $X_1$ exhibits only one unstable eigenvalue, namely $\lambda^\epsilon_3(X_1)$. \textcolor{black}{Using latter key result from Section \ref{Hypothesis1} enables to elegantly justify the asymptotics of the heteroclinic front as $\xi \longrightarrow -\infty$.}
\item \& \,(d) Similarly, one can show for the heteroclinic back $\gamma_b^\epsilon(\xi)$ the convergence along the principle stable and unstable eigenvectors $e_2(X_1)$ and $e_3(X_2)$ as $\xi \longrightarrow +\infty$ and $\xi \longrightarrow -\infty$ respectively. 
\end{enumerate}
\item The idea is to apply the strong $\lambda$-lemma from \cite{BoDeng88} to prove the strong inclination property for vanishing $\epsilon$ and then use robustness to show it for $\epsilon>0$ sufficiently small. To do so, we will need the following lemma:
\begin{lemma}
For $r \in (\frac{2}{3},\beta)\cup(r_0,+\infty)$ the Melnikov integrals with respect to $u$ do not vanish:
\begin{equation}
\frac{\partial Q_j}{\partial u}(u_j(r),D_0(r),\mu_0(r);r) \neq 0,
\label{in123123}
\end{equation}
where $j=f,b$.  
\label{Lemma Melnnonvan}
\end{lemma}
\begin{proof}
As a first case, we show this result along the heteroclinic front (i.e. for $j=f$). From \textcolor{black}{(R\ref{S3.5.1})}, the derivative $\frac{\partial Q_f}{\partial u}$ is given by:
\begin{equation}
    \frac{\partial Q_f}{\partial u}(u_f(r),D_0(r),\mu_0(r);r)=\frac{q_{f,+}(r)^2}{D_0(r)}\widetilde{M_f}(r)
    \label{derdef1}.
\end{equation}
We showed in Section \ref{H3} that $\widetilde{M_f}(r)>0$ for $r\in (\frac{2}{3},\beta)$ in particular:
\begin{equation}
    \forall \,\,r\in (\frac{2}{3},\beta): \widetilde{M_f}(r)\neq 0.
    \label{neq1}
\end{equation}
In addition, according to the \textcolor{black}{asymptotic} behaviour of $\widetilde{M_f}(r)$ as $r \longrightarrow +\infty$ (cf. Theorem \ref{assbehMftilde}):
\begin{equation}
    \exists\,\,r_0>\frac{2}{3}\,\,\,\, \forall\,\,r\in(r_0,+\infty): \widetilde{M_f}(r) < 0.
\end{equation}
So in particular:
\begin{equation}
    \forall\,\,r\in(r_0,+\infty): \widetilde{M_f}(r) \neq 0.
    \label{neq2}
\end{equation}
Now with equation \eqref{expqfpm}, we immediately obtain that $q_{f,+}(r)>1$ so in particular:
\begin{equation}
    \forall \,\,r\in (\frac{2}{3},\beta)\cup(r_0,+\infty): q_{f,+}(r)^2\neq 0.
    \label{neq3}
\end{equation}
Inserting these non-vanishing results \eqref{neq1}, \eqref{neq2} and \eqref{neq3} into expression \eqref{derdef1} yields:
\begin{equation}
    \forall \,\,r\in (\frac{2}{3},\beta)\cup(r_0,+\infty): \frac{\partial Q_f}{\partial u}(u_f(r),D_0(r),\mu_0(r);r)\neq 0.
\end{equation}\\
Next we show the result along the heteroclinic back (i.e. for $j=b$). The Melnikov integral $\frac{\partial Q_b}{\partial u}(u_b(r),D_0(r),\mu_0(r);r)$ can be computed explicitly, cf. \textcolor{black}{(R\ref{S3.5.1})}, as:
\begin{equation}
\begin{split}
        \frac{\partial Q_b}{\partial u}(u_b(r),D_0(r),\mu_0(r);r)=-\frac{1}{D_0(r)}\int_{-\infty}^{+\infty}e^{-\frac{\mu_0(r)+u_b(r)}{D_0(r)}\xi}s_b(\xi;r)(s_b(\xi;r)- q_b(\xi;r))d\xi.
\end{split}
\end{equation}

$q_b(\xi;r)$ connects $q_{b,+}(r)>0$ to 0 and is monotonically decreasing, hence:
\begin{equation}
    \forall\,\,r>\frac{2}{3}: \frac{\partial Q_b}{\partial u}(u_b(r),D_0(r),\mu_0(r);r)<0.
    \label{Qbu}
\end{equation}
So in particular:
\begin{equation}
    \forall\,\,r\in(\frac{2}{3},\beta)\cup(r_0,+\infty): \frac{\partial Q_b}{\partial u}(u_b(r),D_0(r),\mu_0(r);r) \neq 0.
\end{equation}
\end{proof}
\textcolor{black}{Inequality \eqref{in123123} is decisive, since it implies in particular that $W_{\boldsymbol{\alpha}_0}^s(K_{i,0})$ and $W_{\boldsymbol{\alpha}_0}^u(K_{j,0})$ intersect transversely:}

\textcolor{black}{\begin{corollary}
    For $r \in (\frac{2}{3},\beta)\cup(r_0,+\infty)$, we have:
    \begin{enumerate}
        \item a transverse crossing of the stable and unstable manifolds $W_{\boldsymbol{\alpha}_0}^s(K_{1,0})$ and $W_{\boldsymbol{\alpha}_0}^u(K_{2,0})$
        \item a transverse crossing of the stable and unstable manifolds $W_{\boldsymbol{\alpha}_0}^s(K_{2,0})$ and $W_{\boldsymbol{\alpha}_0}^u(K_{1,0})$
    \end{enumerate}
    \label{cortransvross}
\end{corollary}}
 Now we apply the strong $\lambda$-lemma (cf. \cite{BoDeng88}) to these manifolds, which shows that the strong inclination conditions \eqref{cond1} and \eqref{cond2} are met for $\epsilon=0$. The extension of this result to an interval $(0,\epsilon_0(r))$ with $\epsilon_0(r)>0$ sufficiently small follows directly from the robustness of the strong inclination property.
 \textcolor{black}{This step is key to extend the domain of application after having shown the property for one specific value of $\epsilon$}. This shows that the heteroclinic loop established in Section \ref{existenceheteroclinicloop} is non-degenerate for $r \in (\frac{2}{3},\beta)\cup(r_0,+\infty)$ and $\epsilon>0$ taken sufficiently small. 
\end{enumerate}

\subsection{Hypothesis \texorpdfstring{$H_4$}{Lg}: Linear independency of the Melnikov integrals \texorpdfstring{$\nabla Q_f$ \& $\nabla Q_b$}{Lg} and convergence of \texorpdfstring{$\gamma_j^\epsilon(\xi)$ \& $\psi_j(\xi)$}{Lg} for \texorpdfstring{$j\in\{f,b\}$}{Lg}}
We need to show that $\nabla Q_f$ and $\nabla Q_b$ are linearly independent, where the gradient operator $\nabla$ is taken with respect to the relevant parameters $\mu$ and $D$, hence:
\begin{equation}
\nabla Q_f, \nabla Q_b \in \mathbb{R}^2.
\end{equation}
\begin{theorem}
    $\nabla Q_f$ and $\nabla Q_b$ are linearly independent (and in particular non-zero).
    \label{linind}
\end{theorem}
\begin{proof}
One can derive (cf. \cite{Holmes1980}, \cite{Robinson77}, \cite{Melnikov63}) the following expression for the derivatives of $Q_j$, $j\in \{f,b\}$, with respect to $i\in \{\mu,D\}$ leading to the Melnikov integrals:
\begin{equation}
    \frac{\partial Q_j}{\partial i}(\boldsymbol{\alpha}_0(r);r)=-\int_{- \infty}^{+ \infty}\psi_j(\xi;r)\cdot \frac{\partial \mathcal{F}}{\partial i}(X_j(\xi);\boldsymbol{\alpha}_0(r);r)d\xi
    \label{Mint}
\end{equation}
with 
\begin{equation}
    \partial_\xi X=\mathcal{F}(X;\boldsymbol{\alpha},r) 
    \label{abstract form}
\end{equation}
the dynamical system \eqref{eqdiff} written in abstract form, where $\mathcal{F}:\mathbb{R}^3 \times \mathcal{V}\times (\frac{2}{3},+ \infty)\longrightarrow \mathbb{R}^3$ and $\mathcal{V} \subset \mathbb{R}^3$ is a small neighborhood of $\boldsymbol{\alpha}_0$. The function $\psi_j=\psi_j(\xi;r)$ denotes the solution of the so-called $\textit{adjoint variational equation}$:
\begin{equation}
    \partial_\xi \psi=- \Big(\partial_X\mathcal{F}(X_j(\xi);\boldsymbol{\alpha}_0(r);r)\Big)^\top \psi
    \label{adjvareq}
\end{equation}
about the heteroclinic $X_j(\xi)$ with initial condition at $\xi=0$:
\begin{equation}
    \psi_j(0;r)=(s_j'(0;r),-q_j'(0;r),0)=:-e_j.
    \label{initcond}
\end{equation}

Assume that $\nabla Q_f=\begin{pmatrix}
\frac{\partial Q_f}{\partial D}(\boldsymbol{\alpha}_0(r);r)\\
\frac{\partial Q_f}{\partial \mu}(\boldsymbol{\alpha}_0(r);r) \\
\end{pmatrix}$ and $\nabla Q_b=\begin{pmatrix}
\frac{\partial Q_b}{\partial D}(\boldsymbol{\alpha}_0(r);r)\\
\frac{\partial Q_b}{\partial \mu}(\boldsymbol{\alpha}_0(r);r) \\
\end{pmatrix}$ are linearly dependent for some $r>\frac{2}{3}$. Then the determinant of the corresponding $2 \times 2$ matrix with columns $\nabla Q_f$ and $\nabla Q_b$ would vanish,
which implies 
\begin{equation}
    \widehat{M}(r)=0.
\end{equation}
However, this contradicts the non-zeroness condition for $\widehat{M}$ which holds for all $r>\frac{2}{3}$ as shown in Section \ref{existenceheteroclinicloop}, and $\nabla Q_f$ and $\nabla Q_b$ are linearly independent. 
\end{proof}
Now we look at the convergence of $\psi_j$ and $\gamma_j^\epsilon$ and prove the following results:
\begin{theorem}
     For $j = f, b$, $\psi_j(\xi)$ and $\gamma_j^\epsilon(\xi)$ converge \textcolor{black}{as $\xi \longrightarrow +\infty$} to zero and the equilibria respectively in the Reynolds number regime $r\in (\frac{2}{3},\beta)$ with $\beta\approx 0.72946$. $\psi_j$ denotes the solution of the adjoint variational equation \eqref{adjvareq} with initial condition \eqref{initcond} as introduced in the proof of Theorem \ref{linind}.
\end{theorem}
\begin{proof}
        According to equations \eqref{limit1} and \eqref{limit2}, $\gamma_f^\epsilon(\xi) \underset{\xi \longrightarrow \pm \infty}{\longrightarrow} X_{2,1}$ and $\gamma_b^\epsilon(\xi) \underset{\xi \longrightarrow \pm \infty}{\longrightarrow} X_{1,2}$ for every $r>\frac{2}{3}$. To show the convergence to 0 of $\psi_j$ we use the following explicit expressions for $r>\frac{2}{3}$, which can be readily checked following \cite{EKR2022}   
        \begin{equation}
            \psi_f(\xi;r)=e^{-\frac{\mu_0(r)+2}{D_0(r)}\xi}(\dot{s}_f(\xi),-s_f(\xi;r),0),\,\,\,\,\,\,\,\,\psi_b(\xi;r)=e^{-\frac{\mu_0(r)+u_b(r)}{D_0(r)}\xi}(\dot{s}_b(\xi),-s_b(\xi;r),0).
            \label{expressionpsifb}      
        \end{equation}
      
    Bounded orbits $(\widetilde{q}(\xi),\widetilde{s}(\xi),\widetilde{u}(\xi))$ in the dynamical system \eqref{eqdiff} directly correspond to traveling-wave solutions to the Barkley model \eqref{primary}. Hence, it follows that the $3$-dimensional vectors $(\dot{s}_f(\xi),-s_f(\xi;r),0)$ and $(\dot{s}_b(\xi),-s_b(\xi;r),0)$ in \eqref{expressionpsifb} are bounded. Note that the function $f$, which appears in the expression of $\dot{s}$ (cf. \eqref{eqdiff}): 
    \begin{equation}
        f(q,u;r)=q(r+u-2-(r+0.1)(q-1)^2),
    \end{equation}
    is bounded for fixed $r$, since $q$ and $u$ are bounded as components of bounded orbits. \\
    \newline
For the asymptotical behaviour of $\psi_f$ as $\xi \longrightarrow + \infty$, we observe that since $\mu_0(r)>-\frac{8}{5}$ (see R(\ref{T2.2})) and $D_0(r)>0$, the quotient $\frac{\mu_0(r)+2}{D_0(r)}$ is positive and hence for $r>\frac{2}{3}$
    \begin{equation}
         e^{-\frac{\mu_0(r)+2}{D_0(r)}\xi} \underset{\xi \longrightarrow + \infty}{\longrightarrow} 0.
    \end{equation}
    Together with the boundedness of $(\dot{s}_f(\xi),-s_f(\xi;r),0)$ it follows immediately
    \begin{equation}
       \forall \,\, r>\frac{2}{3}: \lim_{\xi \longrightarrow +\infty} \psi_f(\xi;r)=0.
    \end{equation}
    Now for the convergence to 0 of $\psi_b$ as $\xi \longrightarrow + \infty$, we need to have according to \eqref{expressionpsifb} (since $D_0(r)>0$)
    \begin{equation} 
        \mu_0(r)+u_b(r)>0,
        \label{cond}
    \end{equation} 
    which is not necessarily the case for all values of $r>\frac{2}{3}$, since in general we only have $\mu_0(r)>-\frac{8}{5}$ and $u_b(r)>\frac{6}{5}$, see \textcolor{black}{(R\ref{R1})}. However, from the numerical computation of $\mu_0=\mu_0(r)$ in \cite{EKR2022} 2.2 we derive that condition \eqref{cond} is fulfilled for $r \in (\frac{2}{3}, \beta)$, since then $\mu_0(r)>-\frac{6}{5}$.\\
    Hence, in the relevant parameter range $r \in (\frac{2}{3}, \beta)$ we have, similarly to above, that the quotient $\frac{\mu_0(r)+u_b(r)}{D_0(r)}$ is positive and hence
    \begin{equation}
        e^{-\frac{\mu_0(r)+u_b(r)}{D_0(r)}\xi} \underset{\xi \longrightarrow + \infty}{\longrightarrow} 0.
    \end{equation}
    Again, together with the boundedness argument -- this time for the 3-dimensional vector $(\dot{s}_b(\xi),-s_b(\xi;r),0)$ -- it yields:
    \begin{equation}
        \forall\,\,r \in (\frac{2}{3}, \beta): \lim_{\xi \longrightarrow +\infty} \psi_b(\xi;r)=0,
    \end{equation}
    which finishes the proof; as \textcolor{black}{shown in Section \ref{existenceheteroclinicloop}} we emphasize that some explicit integrals have been evaluated numerically in the proof but direct interval arithmetic would validate the signs of these integrals easily as the formulas are fully explicit.
\end{proof}

\subsection{Hypothesis \texorpdfstring{$H_5$}{Lg}: Non-vanishing limits and twist of both heteroclinic orbits \texorpdfstring{$\gamma_f^\epsilon$}{Lg} and \texorpdfstring{$\gamma_b^\epsilon$}{Lg}}
\label{Hypothesis H5}
The first part of Hypothesis ($H_5$) in \cite{Sandstede98} supposes the strong inclination property, as the orthogonality relations $\psi_f(\xi) \perp T_f$ and $\psi_b(\xi) \perp T_b$ are fulfilled, with $T_f$ and $T_b$ the sums of tangent spaces: 
\begin{equation}
    T_f:=T_{\gamma_f^\epsilon(\xi)}W^u_{\alpha(\epsilon;r)}(X_1)+T_{\gamma_f^\epsilon(\xi)}W_{\alpha(\epsilon;r)}^s(X_2),
\end{equation}
\begin{equation}
    T_b:=T_{\gamma_b^\epsilon(\xi)}W_{\alpha(\epsilon;r)}^u(X_2)+T_{\gamma_b^\epsilon(\xi)}W_{\alpha(\epsilon;r)}^s(X_1).
\end{equation}
 As we showed in Section \ref{H3} when proving the non-degeneracy of the heteroclinic loop, the strong inclination conditions are met for $r\in(\frac{2}{3},\beta) \cup (r_0, +\infty)$ with $\beta$  as in equation $\eqref{beta}$ and $r_0$ as in Lemma \eqref{Lemma Melnnonvan}.\\

The second part of Hypothesis ($H_5$) in \cite{Sandstede98} assumes that both heteroclinic orbits $\gamma_f^\epsilon$ and $\gamma_b^\epsilon$ are twisted. The double twisted regime occurs in the Barkley model \eqref{primary} for intermediate Reynolds numbers cf. \textcolor{black}{(R\ref{L3.7})}. More precisely,
for $r \in (\frac{2}{3},\beta)$ we have the following behaviour: 
\begin{enumerate}
    \item The heteroclinic front $\gamma_f^\epsilon$ is twisted. Hence, the principal eigenvectors $e_2(X_1)$ and $e_3(X_2)$ point to opposite sides of the tangent space $T_{\gamma_f^\epsilon(\xi)}W_{\alpha(\epsilon;r)}^s(X_2)$ as $\xi \longrightarrow -\infty$ and $\xi \longrightarrow +\infty$, respectively. 
    \item The heteroclinic back $\gamma_b^\epsilon$ is also twisted. Hence, $e_2(X_2)$ and $e_3(X_1)$ point to opposite sides of $T_{\gamma^\epsilon_b(\xi)} W_{\alpha(\epsilon;r)}^s(X_1)$ as $\xi \longrightarrow -\infty$ and $\xi \longrightarrow +\infty$, respectively.
\end{enumerate}
The twist of the heteroclinic back $\gamma^\epsilon_b$ follows directly from the inequality:
\begin{equation}
    \frac{\partial Q_b}{\partial u}(\boldsymbol{\alpha}_0(r);r)<0,
\end{equation}
which holds for all $r>\frac{2}{3}$ (cf. equation \eqref{Qbu}), hence the heteroclinic back is twisted for all these Reynolds number regimes. Similarly, for the heteroclinic front $\gamma^\epsilon_f$, the twist is a direct consequence of:
\begin{equation}
    \frac{\partial Q_f}{\partial u}(\boldsymbol{\alpha}_0(r);r)>0,
\end{equation}
which is satisfied only for $r \in (\frac{2}{3},\beta)$ (cf. equation \eqref{derdef1} and Section \ref{H3}). Consequently, a double twisted regime shows up for $r \in (\frac{2}{3}, \beta)$ and we only have a single twisted regime for $r \in (\beta, +\infty)$.

\subsection{Hypothesis \texorpdfstring{$H_6$}{Lg}: Positivity of the scalar products \texorpdfstring{$\langle w_j^\pm,v_j^\pm\rangle$}{Lg} for \texorpdfstring{$j\in\{f,b\}$}{Lg}}
\label{last section}
\textcolor{black}{We denote as follows the following limits}
\begin{equation}
    v_j^-:=\lim_{\xi \longrightarrow -\infty}e^{-\lambda_3(X_{i_j})\xi}\dot{\gamma}_j(\xi),\,\,\,\,\,\,v_{k}^+:=\lim_{\xi \longrightarrow +\infty}e^{-\lambda_2(X_{i_k})\xi}\dot{\gamma}_j(\xi),
    \label{firstlimit}
\end{equation}

\begin{equation}
    w_j^+:=\lim_{\xi \longrightarrow -\infty}e^{\lambda_2(X_{i_j})\xi}\psi_j(\xi),\,\,\,\,\,\, w_{k}^-:=\lim_{\xi \longrightarrow +\infty}e^{\lambda_3(X_{i_k})\xi}\psi_j(\xi)
    \label{defwk}
\end{equation}

for \textcolor{black}{$j=f$ and $k=b$ (or vice-versa $j=b$ and $k=f$), with}  $i_f=1$, $i_b=2$. $\psi_j$ denote the bounded solutions of the adjoint variational equation \eqref{adjvareq} with initial condition \eqref{initcond}. 
In this section, we show that the scalar products $\langle w_j^-,v_j^-\rangle$ and $\langle w_j^+,v_j^+\rangle$ are positive for $\epsilon>0$ sufficiently small and $j\in\{f,b\}$. We will explicitly verify the positivity assumption under the conditions of Lemma 3.3 in \cite{EKR2022}:
\begin{equation}
    2+\mu=\frac{1}{2}\sqrt{2D(r+0.1)}(q_{f,+}(r)-2q_{f,-}(r)),
    \label{Con1}
\end{equation}
\begin{equation}
    u_b(r)+\mu=-\frac{1}{2}\sqrt{2D(r+0.1)}(q_{b,+}(r)-2q_{b,-}(r)),
    \label{Con2}
\end{equation}
\textcolor{black}{which are just required matching conditions for the existence of a heteroclinic front and a heteroclinic back.}
 \begin{figure}[H]
    \centering
    \includegraphics[width=14cm]{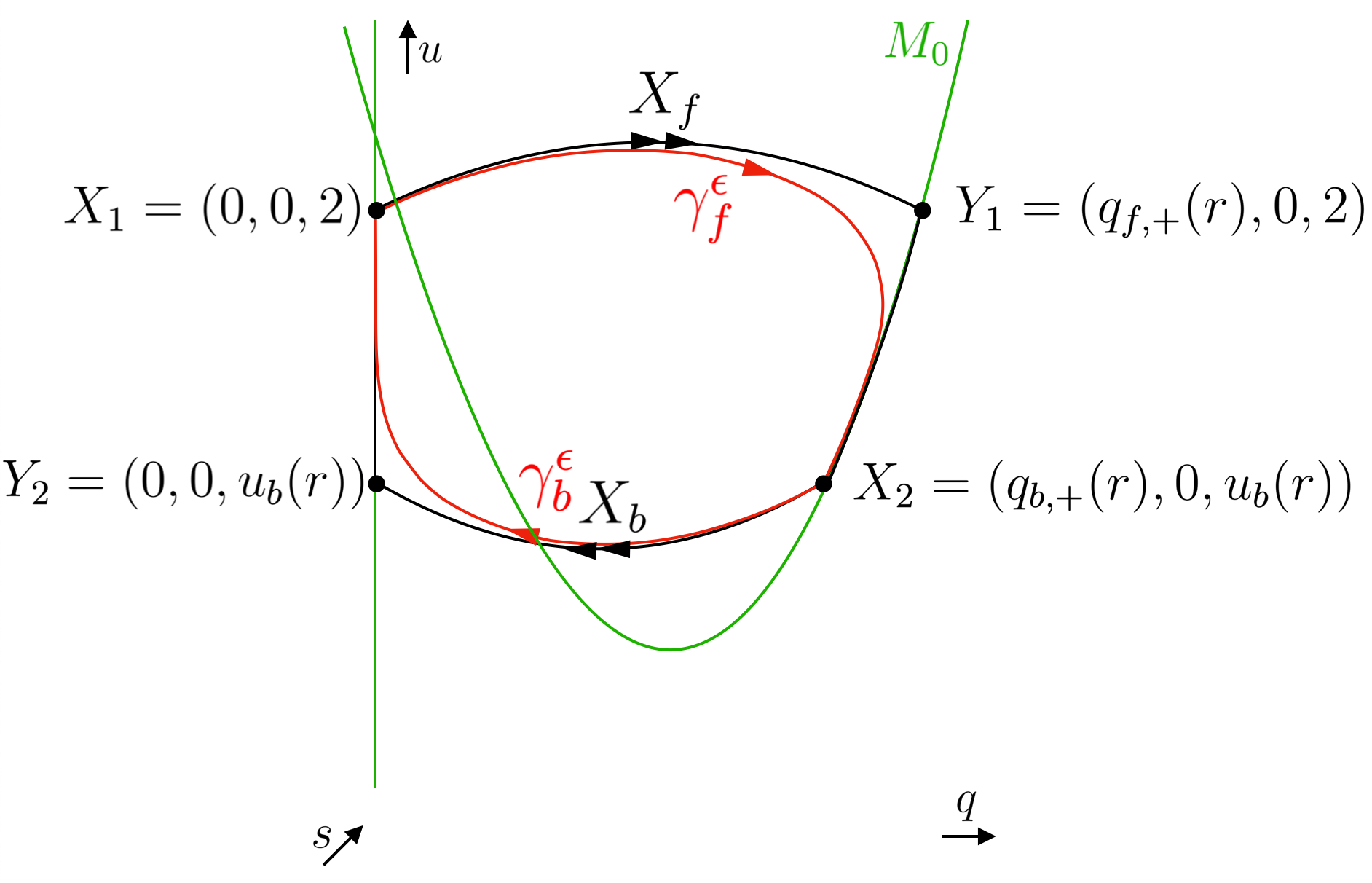}\\[3mm] 
    \caption{Representation of the singular heteroclinic loop (black) and actual heteroclinic connections (red) between the equilibria $X_1$ and $X_2$. The singular heteroclinic loop is made up of the heteroclinic connections $X_f$ and $X_b$ in the fast subsystem \eqref{fastss} and orbit segments lying on the manifold $M_0$ in the slow subsystem \eqref{slowsubsys}.}
    \label{fig: Singloop}
    \end{figure}
 Under conditions \eqref{Con1} and \eqref{Con2} the fast subsystem \eqref{fastss} (i.e. $\epsilon=0$) has for $r>\frac{2}{3}$ a heteroclinic front solution within the layer $u=2$, and a heteroclinic back solution within the layer $u=u_b(r)$, cf. \cite{EKR2022}:
\begin{equation}
X_f(\xi)=(q_f(\xi;r),s_f(\xi;r),u_f(r))=\Big(q_{f,+}(r)\phi\Big(q_{f,+}(r)\sqrt{\frac{r+0.1}{D}}\xi\Big),\dot{q}_f(\xi;r),2\Big)
\label{gammafxi}
\end{equation}
\begin{equation}
X_b(\xi)=(q_b(\xi;r),s_b(\xi;r),u_b(r))=\Big(q_{b,+}(r)\phi\Big(-q_{b,+}(r)\sqrt{\frac{r+0.1}{D}}\xi\Big),\dot{q}_b(\xi;r),u_b(r)\Big).
\label{la}
\end{equation}
$X_f$ connects the hyperbolic saddles $X_1=(0,0,2)$ with $Y_1=(q_{f,+}(r),0,2)$, and $X_b$ the hyperbolic saddles $X_2=(q_{b,+}(r),0,u_b(r))$ with $Y_2=(0,0,u_b(r))$ (cf. Figure \eqref{fig: Singloop}).\\

Note that no actual heteroclinic connections between $X_1$ and $X_2$ are contained in the singular heteroclinic loop, even for $\epsilon=0$. Indeed, the sharpness of the edges at $Y_1$ and $Y_2$ (cf. Figure \eqref{fig: Singloop}) is incompatible with the smoothness of heteroclinic connections. However, the existence of an actual heteroclinic loop -- consisting of $\gamma_f^\epsilon$ and $\gamma_b^\epsilon$ and connecting $X_1$ and $X_2$ --, which lies in the vicinity of the singular one, can be proved for $\epsilon>0$ sufficiently small, cf. \cite{EKR2022} Section 3.4. \\ \\
\textcolor{black}{In the next steps, we will need the expressions for some derivatives, which we summarize in the following lemma. \begin{lemma}
The first and second derivatives with respect to $\xi$ of $q_f(\xi;r)$ and $q_b(\xi;r)$ are given by:
\begin{enumerate}
    \item $\forall\,\,\xi \in \mathbb{R}, \forall\,\,r>\frac{2}{3}: s_f(\xi;r)=\frac{1}{2}q_{f,+}(r)^2 \sqrt{\frac{2(r+0.1)}{D}}\frac{e^{-\frac{1}{2}\sqrt{2}q_{f,+}(r)\sqrt{\frac{r+0.1}{D}}\xi}}{\Big(1+e^{-\frac{1}{2}\sqrt{2}q_{f,+}(r)\sqrt{\frac{r+0.1}{D}}\xi}\Big)^2}$
    \item $\forall\,\,\xi \in \mathbb{R}, \forall\,\,r>\frac{2}{3}: \dot{s}_f(\xi;r)=\frac{q_{f,+}(r)^3}{2}\frac{r+0.1}{D}e^{-\frac{1}{2}\sqrt{2}q_{f,+}(r)\sqrt{\frac{r+0.1}{D}}\xi} \frac{e^{-\frac{1}{2}\sqrt{2}q_{f,+}(r)\sqrt{\frac{r+0.1}{D}}\xi}-1}{\Big(1+e^{-\frac{1}{2}\sqrt{2}q_{f,+}(r)\sqrt{\frac{r+0.1}{D}}\xi}\Big)^3}$
    \item $\forall\,\,\xi \in \mathbb{R}, \forall\,\,r>\frac{2}{3}: s_b(\xi;r)= -\frac{1}{2}q_{b,+}(r)^2\sqrt{\frac{2(r+0.1)}{D}}\frac{e^{\frac{1}{2}\sqrt{2}q_{b,+}(r)\sqrt{\frac{r+0.1}{D}}\xi}}{\Big(1+e^{\frac{1}{2}\sqrt{2}q_{b,+}(r)\sqrt{\frac{r+0.1}{D}}\xi}\Big)^2}$
    \item $\forall\,\,\xi \in \mathbb{R}, \forall\,\,r>\frac{2}{3}: \dot{s}_b(\xi;r)=\frac{q_{b,+}(r)^3}{2}\frac{r+0.1}{D}e^{\frac{1}{2}\sqrt{2}q_{b,+}(r)\sqrt{\frac{r+0.1}{D}}\xi}\frac{e^{\frac{1}{2}\sqrt{2}q_{b,+}(r)\sqrt{\frac{r+0.1}{D}}\xi}-1}{\Big(1+e^{\frac{1}{2}\sqrt{2}q_{b,+}(r)\sqrt{\frac{r+0.1}{D}}\xi}\Big)^3}$.
\end{enumerate}
\label{Lemma derivatives}
\end{lemma}}
\begin{proof}
\textcolor{black}{Immediate after computation using:}

\begin{enumerate}
\item \textcolor{black}{$\forall\,\,\xi \in \mathbb{R}, \forall\,\,r>\frac{2}{3}: s_f(\xi;r)=\frac{d}{d \xi}q_f(\xi;r)$ and $\forall\,\,\xi \in \mathbb{R}, \forall\,\,r>\frac{2}{3}: s_b(\xi;r)=\frac{d}{d \xi}q_b(\xi;r)$}
\item \textcolor{black}{$\forall\,\,\xi \in \mathbb{R}, \forall\,\,r>\frac{2}{3}: q_f(\xi;r)=q_{f,+}(r)\phi\Big(q_{f,+}(r)\sqrt{\frac{r+0.1}{D}}\xi\Big)$ from equation \eqref{gammafxi}, recall $\forall\,\,\chi \in \mathbb{R}: \phi(\chi)=\frac{1}{1+e^{-\frac{1}{2}\sqrt{2}\chi}}$ from \eqref{defphi}} 
\item \textcolor{black}{$\forall\,\,\xi \in \mathbb{R}, \forall\,\,r>\frac{2}{3}: q_b(\xi;r)=q_{b,+}(r)\phi\Big(-q_{b,+}(r)\sqrt{\frac{r+0.1}{D}}\xi\Big)$ from equation \eqref{la}.}
\end{enumerate}
\end{proof}
\begin{enumerate}
    \item \label{wbminusvbminus} \underline{Computing $\langle w_b^-,v_b^-\rangle$}: 
    \begin{enumerate}
        \item  We first show the result for $\epsilon=0$ (corresponding to the fast subsystem \eqref{fastss}) and then extend it to $\epsilon>0$
    sufficiently small. As the first case $\epsilon=0$, we have: 
         \textcolor{black}{\begin{equation}     
         \begin{split}
         \langle w_b^-, v_b^- \rangle   &=\langle \lim_{\xi \longrightarrow +\infty}e^{\lambda_3(X_2)\xi}\psi_f(\xi) , \lim_{\xi \longrightarrow +\infty}e^{\lambda_3(X_2)\xi}\dot{X}_b(-\xi) \rangle.
    \end{split}
     \end{equation}}
     Using the linearity and the continuity of the scalar product, we can write:
     \textcolor{black}{\begin{equation}
     \begin{split}
     \langle w_b^-, v_b^- \rangle&= \lim_{\xi \longrightarrow +\infty}e^{2\lambda_3(X_2)\xi} \langle \psi_f(\xi) , \dot{X}_b(-\xi) \rangle.
     \end{split}
     \end{equation}}
     With the definition of the heteroclinic connection $X_b$ and expression \eqref{expressionpsifb} for the solution $\psi_f$ of the adjoint variational equation \ref{adjvareq} with initial condition \ref{initcond}, we obtain:
     \begin{equation}
     \begin{split}
     \langle w_b^-, v_b^- \rangle
      &=\lim_{\xi \longrightarrow +\infty}e^{\big(2\lambda_3(X_2)-\frac{\mu_0(r)+2}{D_0(r)}\big)\xi} \big(\dot{s}_f(\xi) s_b(-\xi) - s_f(\xi) \dot{s}_b(-\xi) \big)
     \end{split}
     \label{li0}
     \end{equation}
with $s_f$, $s_b$ and $u_b$ defined in equations \eqref{gammafxi} and \eqref{la}.\textcolor{black}{To simplify notations further on, we denote by: \begin{equation}
    d^0_{b^-}(\xi):=\dot{s}_f(\xi) s_b(-\xi) - s_f(\xi) \dot{s}_b(-\xi)
\end{equation} the second factor appearing in above equation.}\\ \\ \textcolor{black}{By contradiction, assume that \begin{equation}
    \langle w_b^-, v_b^- \rangle=:a_b^- < 0.
    \label{assumb-}
\end{equation}} Note that the scalar products do not vanish according to hypotheses $H_1$ and $H_6$, cf.~Sandstede's remark in the statement of $H_6$~\textcolor{black}{\cite{Sandstede98} where this conclusion is discussed in more detail.} \textcolor{black}{Then: 
\begin{equation}
    \forall\,\,\widetilde{\epsilon }> 0, \,\,\exists \,\,\ \xi_{0,b^-} \in \mathbb{R},\,\,\,\,\forall\,\,\xi > \xi_{0,b^-}: 
     \vert e^{\big(2\lambda_3(X_2)-\frac{\mu_0(r)+2}{D_0(r)}\big)\xi}d^0_{b^-}(\xi)-a_b^- \vert < \widetilde{\epsilon }
     \label{ab-ineg}
\end{equation}
In particular, for $\widetilde{\epsilon }=\frac{\vert a_b^- \vert}{2}>0$ and $\xi_{0,b^-}$ chosen accordingly such that \ref{ab-ineg} is satisfied, we have: 
\begin{equation}
    \forall\,\,\xi>\xi_{0,b^-}:e^{\big(2\lambda_3(X_2)-\frac{\mu_0(r)+2}{D_0(r)}\big)\xi}d^0_{b^-}(\xi) \in (a_b^--\widetilde{\epsilon }, a_b^-+\widetilde{\epsilon })=(\frac{3}{2}a_b^-, \frac{a_b^-}{2})
\end{equation}
which is a subset of $\mathbb{R}^-\backslash \{0\}$. Due to the positiveness of the exponential prefactor, we have in particular: 
\begin{equation}
    \forall\,\,\xi>\xi_{0,b^-}:d^0_{b^-}(\xi) <0
    \label{xi0b-}
\end{equation}}

Using \textcolor{black}{expressions 1. to 4. from Lemma \ref{Lemma derivatives}} we compute the following limits as $\xi \longrightarrow +\infty$:
\begin{equation}
    \lim_{\xi \longrightarrow +\infty}s_f(\xi)=0^+,\,\,\,\lim_{\xi \longrightarrow +\infty}\Dot{s}_f(\xi)=0^-,\,\,\, \lim_{\xi \longrightarrow +\infty}s_b(-\xi)=0^-,\,\,\,\lim_{\xi \longrightarrow +\infty}\Dot{s}_b(-\xi)=0^-,
    \label{limA1}
\end{equation}

where $0^\pm$ means that the limit value \textit{zero} is approached from the positive (respectively negative) side, yielding immediately:
\begin{equation}
 \lim_{\xi \longrightarrow +\infty}\textcolor{black}{d^0_{b^-}(\xi)}=0^+.
 \label{lim0}
\end{equation}
 \textcolor{black}{Hence, there exists $\xi_{1,b^-} \in \mathbb{R}$ such that for all $\xi>\xi_{1,b^-}$: 
 \begin{equation}
     d^0_{b^-}(\xi)\geq 0
 \end{equation}
 For $\xi>\max \{\xi_{0,b^-},\xi_{1,b^-}\}$ ($\xi_{0,b^-}$ as in \eqref{xi0b-}) we have as well $d^0_{b^-}(\xi)<0$ as $d^0_{b^-}(\xi)\geq0$, which is a contradiction. Hence our assumption \eqref{assumb-} was false, and as a consequence for $\epsilon=0$:
 \begin{equation}
     \langle w_b^-, v_b^- \rangle>0.
\label{Ineq1eps0}
 \end{equation}}
 \item Now we generalize result \eqref{Ineq1eps0} to $\epsilon>0$ sufficiently small. Let $\gamma^\epsilon_b$ be a heteroclinic back solution of the Barkley model \eqref{eqdiff} with parameter $\epsilon>0$ taken sufficiently small, connecting $X_2$ to $X_1$. We define the components of $\gamma^\epsilon_b$ in the $(q,s,u)$-frame as:
\begin{equation}
    \gamma^\epsilon_b(\xi)=:
\begin{pmatrix}
q_{b}^\epsilon(\xi)  \\
s_{b}^\epsilon(\xi) \\
u_{b}^\epsilon(\xi)
\end{pmatrix}.
\end{equation}
Following the reasoning in (a), we need to determine the limit of
\textcolor{black}{\begin{equation}
d^\epsilon_{b^-}(\xi):=\dot{s}_f(\xi)s_b^\epsilon(-\xi)-s_f(\xi)\dot{s}_b^\epsilon(-\xi)
\end{equation}
as $\xi$ tends to $+\infty$,} the exponential prefactor
\begin{equation}
    e^{\big(2\lambda_3(X_2)-\frac{\mu_0(r)+2}{D_0(r)}\big)\xi}
\end{equation}
remaining unchanged (cf. equation \eqref{li0}). From \textcolor{black}{Lemma \ref{Lemma derivatives} point 3} we obtain for the $s$-component of the singular heteroclinic back:
\begin{equation}
    \forall\,\,\xi \in \mathbb{R}: s_b(\xi)<0.
    \label{sbinf}
\end{equation}
For $\epsilon>0$ sufficiently small, $\gamma_b^\epsilon$ lies in the vicinity of the singular heteroclinic back, which yields:
\begin{equation}
    \exists \,\,\xi_0 \in \mathbb{R}\,\,\,\,\forall\,\,\xi < \xi_0: s_b^\epsilon(\xi)<0.
    \label{ineq2}
\end{equation}
And since 
\begin{equation}
    \lim_{\xi \longrightarrow -\infty}\gamma_b^\epsilon(\xi)=X_2=(q_{b,+}(r),0,u_b(r)),
    \label{hequ}
\end{equation}
we immediately have:
\begin{equation}
    \lim_{\xi \longrightarrow +\infty}s_b^\epsilon(-\xi)=0,
\end{equation}
which means together with inequality 
\eqref{ineq2} that
\begin{equation}
    \lim_{\xi \longrightarrow +\infty}s_b^\epsilon(-\xi)=0^-
    \label{sblimeps}
\end{equation}
for $\epsilon>0$ sufficiently small.\\
\newline
Now we compute the limit of the derivative of $s_b^\epsilon (-\xi)$ as $\xi \longrightarrow +\infty$, which satisfies according to the dynamical system \eqref{eqdiff} the differential equation:
\begin{equation}
    \Dot{s}_b^\epsilon=\frac{1}{D}\Big((\underbrace{u_b^\epsilon+\mu}_{A})\underbrace{s_b^\epsilon}_{B}-\underbrace{q_b^\epsilon}_{C}\big(\underbrace{r+u_b^\epsilon-2-(r+0.1)(q_b^\epsilon-1)^2}_{D}\big)\Big).
    \label{ABCD}
\end{equation}
\par
\begingroup
\leftskip=1cm 
\noindent 
\underline{Term $A$:}
\begin{equation}
\lim_{\xi \longrightarrow +\infty}u_b^\epsilon(-\xi;r)+\mu=u_b(r)-\zeta-c>0
\end{equation}
for $\zeta>0$ sufficiently small, since $c<u_b(r)$ (cf. beginning of section \ref{proofbark}, hypotheses).\\
\par
\endgroup
\par
\begingroup
\leftskip=1cm 
\noindent 
\underline{Term $B$:} cf. equation \eqref{sblimeps}.\\
\par
\endgroup
\par
\begingroup
\leftskip=1cm 
\noindent 
\underline{Term $C$:}
\begin{equation}
\lim_{\xi \longrightarrow +\infty}q_b^\epsilon(-\xi;r)=q_{b,+}(r)>0
\end{equation}
(cf. equations \eqref{eqb+} and \eqref{hequ}).\\
\par
\endgroup
\par
\begingroup
\leftskip=1cm 
\noindent 
\underline{Term $D$:} Define
\begin{equation}
    \widetilde{f}(q,u;r):=\frac{1}{q}f(q,u;r),
\end{equation}
which is well-defined in the vicinity of $X_2$, far from the plane $q=0$. The equilibrium $X_2$ lies on the parabola of the critical manifold $M_0$ (cf. Figure \ref{fig: Singloop}), hence 
\begin{equation}
    \widetilde{f}(q_{b,+}(r),u_b(r);r)=r+u_b(r)-2-(r+0.1)(q_{b,+}(r)-1)^2=0
    \label{f01}
\end{equation}
(cf. also \textcolor{black}{(R\ref{R1}))}. The $u$-component of the singular heteroclinic connection $X_b$ stays constant between $X_2$ and $Y_2$ ($u=u_b(r)$), however its $q$-component is decreasing, as $q_b'(\xi)=s_b(\xi)<0$ for $\xi \in \mathbb{R}$ according to \textcolor{black}{Lemma \ref{Lemma derivatives} point 3}. For the derivative of $\widetilde{f}$ with respect to $\xi$, we obtain with constant value of $u$ on the singular connection $X_b$:
\begin{equation}
    \widetilde{f}'(\xi)=-2(r+0.1)(q_b(\xi)-1)q_b'(\xi).
    \label{f'}
\end{equation}
We immediately get the existence of a $\xi_{2,b^-} \in \mathbb{R}$ such that \eqref{f'} is positive for all $\xi<\xi_{2,b^-}$, as 
\begin{equation}
    \lim_{\xi \longrightarrow -\infty}q_b(\xi)=q_{b,+}(r)>1
\end{equation}
(cf. equation \ref{eqb+}) and
\begin{equation}
    \forall \,\,\xi \in \mathbb{R}: s_b(\xi)<0
\end{equation}
cf. \textcolor{black}{Lemma \ref{Lemma derivatives} point 3}. Together with 
\begin{equation}
    \lim_{\xi \longrightarrow -\infty}\widetilde{f}(q_b(\xi),u_b(r);r)=\widetilde{f}(q_{b,+}(r),u_b(r);r)=0
\end{equation}
by continuity of $\widetilde{f}$ in the vicinity of $X_2$ and equation \eqref{f01}, this implies:
\begin{equation}
    \widetilde{f}(q_b(\xi),u_b(r);r)>0
\end{equation}
for all $\xi<\xi_{2,b^-}$, meaning nothing else than all points of $X_b(\xi)$ in a neighborhood of $X_2$ being "inside" the parabola of the manifold $M_0$ in the $(q,u)$-plane (cf. Figure \ref{fig: Singloop}). For $\epsilon>0$ sufficiently small this also holds true for an actual heteroclinic connection $\gamma_b^\epsilon$, since $\gamma_b^\epsilon$ then lies in the vicinity of $X_b$. Rearranging the expression for $\widetilde{f}$ hence yields:
\begin{equation}
    u_b^\epsilon(\xi)>2-r+(r+0.1)(q_b^\epsilon(\xi)-1)^2
\end{equation} 
for $\xi$ in a neighborhood of $-\infty$, and since 
\begin{equation}
    \lim_{\xi \longrightarrow -\infty}\widetilde{f}(q_b^\epsilon(\xi),u_b^\epsilon(\xi);r)=0
\end{equation}
we have:
\begin{equation}
    \lim_{\xi \longrightarrow +\infty}r+u_b^\epsilon(-\xi)-2-(r+0.1)(q_b^\epsilon(-\xi)-1)^2=0^+.
\end{equation}\\
\par
\endgroup
Hence we have with expression \eqref{ABCD} for $\epsilon>0$ sufficiently small:
\begin{equation}
    \lim_{\xi \longrightarrow +\infty}\Dot{s}_b^\epsilon(-\xi)=0^-
    \label{eq163}
\end{equation}
Using the limits \eqref{limA1}, \eqref{sblimeps} and \eqref{eq163} we compute:
\begin{equation}
    \lim_{\xi \longrightarrow +\infty}\textcolor{black}{d^\epsilon_{b^-}(\xi)}=0^+
\end{equation}
and analogously to the case $\epsilon=0$, we conclude for $\epsilon>0$ sufficiently small:
\begin{equation}
\langle w_b^-, v_b^- \rangle>0.
\end{equation}
\end{enumerate}
\item \label{wbplusvbplus} \underline{Computing $\langle w_b^+,v_b^+\rangle$}: \\ \\
    \textcolor{black}{Similarly to the work in \ref{wbminusvbminus}, we have}
    \textcolor{black}{\begin{equation}
     \begin{split}
         \langle w_b^+, v_b^+ \rangle&=\langle \lim_{\xi \longrightarrow +\infty}e^{-\lambda_2(X_2)\xi}\psi_b(-\xi) , \lim_{\xi \longrightarrow +\infty}e^{-\lambda_2(X_2)\xi}\dot{\gamma}_f^\epsilon(\xi) \rangle.
    \end{split}
     \end{equation}} 
     Again, using the linearity and the continuity of the scalar product yields:
    \textcolor{black}{\begin{equation}
     \begin{split}
     \langle w_b^+, v_b^+ \rangle &=\lim_{\xi \longrightarrow +\infty}e^{-2\lambda_2(X_2)\xi} \langle \psi_b(-\xi) , \dot{\gamma}_f^\epsilon(\xi) \rangle.
     \end{split} 
     \end{equation}}
     With the definition of the orbit $\gamma_f^\epsilon$ and expression \eqref{expressionpsifb} for the solution $\psi_b$ of the adjoint variational equation \eqref{adjvareq} with initial condition \eqref{initcond}, we obtain:
    \begin{equation}
     \begin{split}
     \langle w_b^+, v_b^+ \rangle 
      &=\lim_{\xi \longrightarrow +\infty}e^{\big(-2\lambda_2(X_2)+\frac{\mu_0(r)+\mu_b(r)}{D_0(r)}\big)\xi} \big(\dot{s}_b(-\xi) s_f^\epsilon(\xi) - s_b(-\xi) \dot{s}_f^\epsilon(\xi) \big).
     \end{split}
     \label{exp11}
     \end{equation}
     \textcolor{black}{For simplicity we introduce: \begin{equation}
         d^\epsilon_{b^+}(\xi):=\dot{s}_b(-\xi) s_f^\epsilon(\xi) - s_b(-\xi) \dot{s}_f^\epsilon(\xi)
     \end{equation}} 
     Now we investigate the behaviour as $\xi \longrightarrow + \infty$ of $s_f^\epsilon(\xi)$ and $\Dot{s}_f^\epsilon(\xi)$ of the actual heteroclinic front for $\epsilon>0$ sufficiently small. From \textcolor{black}{(R\ref{L3.7})} we know that $W^s_{\boldsymbol{\alpha}(\epsilon;r)}(K_{2,\boldsymbol{\alpha}(\epsilon;r)})$ and $W^u_{\boldsymbol{\alpha}(\epsilon;r)}(K_{1,\boldsymbol{\alpha}(\epsilon;r)})$ intersect transversally along $\gamma_f^\epsilon$. As the heteroclinic loop is non-degenerate, $\gamma^\epsilon_f(\xi)$ is asymptotically tangent for $\xi \longrightarrow +\infty$ to the principal stable eigenvector $e_2(X_2)$ of $X_2=(q_{b,+}(r),0,u_b(r))$, cf. Section \ref{H3}. The stable manifold $W_{\boldsymbol{\alpha}}^s(K_{2,\boldsymbol{\alpha}})$ of the invariant manifold $K_{2,\boldsymbol{\alpha}}$ ($K_{i,\boldsymbol{\alpha}}$ as defined in \textcolor{black}{(R\ref{K10K20})}) is two-dimensional and consists of all orbits (locally) converging to $K_{2,\boldsymbol{\alpha}}$ as $\xi \longrightarrow + \infty$. The unstable manifold $W_{\boldsymbol{\alpha}}^u(K_{1,\boldsymbol{\alpha}})$ of the invariant manifold $K_{1,\boldsymbol{\alpha}}$ is one-dimensional and is made up of all orbits (locally) converging to $K_{1,\boldsymbol{\alpha}}$ as $\xi \longrightarrow - \infty$. \textit{Geometric singular perturbation theory} yields that  $W_{\boldsymbol{\alpha}}^s(K_{2,\boldsymbol{\alpha}})$ coincides for $\epsilon>0$ with the stable manifold $W_{\boldsymbol{\alpha}}^s(X_2)$ of the steady state $X_2$ in system \eqref{eqdiff}, cf. \textcolor{black}{(R\ref{defQfQb})}. The unstable manifold $W^u_{\boldsymbol{\alpha}}(K_{1,\boldsymbol{\alpha}})$ of the invariant manifold $K_{1,\boldsymbol{\alpha}}$ and the stable manifold $W^s_{\boldsymbol{\alpha}}(K_{2,\boldsymbol{\alpha}})$ of the invariant manifold $K_{2,\boldsymbol{\alpha}}$ depend smoothly on $\boldsymbol{\alpha}$ for $\boldsymbol{\alpha}$ close to $\boldsymbol{\alpha_0}$. This prevents any oscillatory behaviour of the actual heteroclinic front for $\epsilon>0$ along the singular one, which arises at $\boldsymbol{\alpha}=\boldsymbol{\alpha}_0$.\\ 

Since 
\begin{equation}
    \lim_{\xi \longrightarrow + \infty}\gamma_f^\epsilon(\xi)=X_2,
    \label{trilim}
\end{equation} $\gamma_f^\epsilon$ lies, for $\epsilon>0$ sufficiently small, in the vicinity as $\xi \longrightarrow + \infty$ of the orbit segment in the slow subsystem \eqref{slowsubsys}, which is located on the critical manifold $M_0$ and connects $Y_1=(q_{f,+}(r),0,2)$ to $X_2=(q_{b,+}(r),0,u_b(r))$, cf. Figure \ref{fig: Singloop} for a graphical representation. This orbit segment is located in the $(q,u)$-plane on the right branch of the parabola (which minimum is reached at $q=1$)
\begin{equation}
    u(q;r)=2+(r+0.1)(q-1)^2-r,
\end{equation}
 since
\begin{equation}
    1<q_{b,+}(r)<q_{f,+}(r).
    \label{setin}
\end{equation}
Relation \eqref{setin} follows immediately from the expressions for $q_{f,+}(r)$ and $q_{b,+}(r)$ from \textcolor{black}{(R\ref{R1})} with the inequality
\begin{equation}
    u_b(r)<2,
\end{equation}
since $u_b(r) \in (\frac{6}{5},\frac{4}{3})$. Any path on the right branch of the parabola connecting a point $X_a=(q_a,0,u_a)$ to another point $X_b=(q_b,0,u_b)$ with $q_a>q_b$ is decreasing in $q$ due to monotonicity. Since the actual heteroclinic front $\gamma_f^\epsilon$ lies in the vicinity for $\epsilon>0$ sufficiently small of the singular one and we could exclude an oscillatory behaviour, the $q$-component of $\gamma_f^\epsilon$ is also decreasing for $\xi \longrightarrow + \infty$, which implies:
\begin{equation}
    \exists \,\,\xi_0 \in \mathbb{R}\,\,\,\,\forall \,\,\xi>\xi_0: s_f^{\epsilon}(\xi)=\Dot{q}_f^{\epsilon}(\xi)<0.
    \label{ineq34}
\end{equation}
From equation \eqref{trilim} we immediately have:
\begin{equation}
     \lim_{\xi \longrightarrow + \infty}s_f^\epsilon(\xi)=0,
     \label{limsfesp0-}
\end{equation}
which yields with inequality \eqref{ineq34}
\begin{equation}
    \lim_{\xi \longrightarrow + \infty}s_f^\epsilon(\xi)=0^-.
    \label{Alimsfesp0-}
\end{equation}

Now we compute the limit as $\xi \longrightarrow +\infty$ of the derivative $\Dot{s}_f^\epsilon(\xi)$  for $\epsilon>0$ sufficiently small. From the dynamical system \eqref{eqdiff}, we know that the derivative satisfies
        \begin{equation}
    \Dot{s}_f^\epsilon(\xi)=\frac{1}{D}\Big((u_f^\epsilon(\xi)+\mu)s_f^\epsilon(\xi)-f(q_f^\epsilon(\xi),u_f^\epsilon(\xi);r)\Big).
        \end{equation}
        We have by continuity of $f$:
        \begin{equation}
    \lim_{\xi \longrightarrow +\infty}f(q_f^\epsilon(\xi),u_f^\epsilon(\xi);r)=f(q_{b,+}(r),u_b(r);r)=0,
    \label{f0}
        \end{equation}
        since $X_2$ lies on the manifold $M_0$. \textcolor{black}{Since $u_f^\epsilon(\xi)$ is bounded for all $\xi>0$},  \eqref{limsfesp0-} and \eqref{f0} yield that $\Dot{s}_f^\epsilon(\xi)$ vanishes for $\xi \longrightarrow +\infty$:
        \begin{equation}
            \lim_{\xi \longrightarrow +\infty}\Dot{s}_f^\epsilon(\xi)=0.
        \end{equation}
As equation \eqref{Alimsfesp0-} holds and $s_f^\epsilon$ cannot show an oscillatory behaviour, we have
\begin{equation}
    \exists\,\,\xi_0 \in \mathbb{R}\,\,\,\,\forall\,\,\xi > \xi_0: \Dot{s}_f^\epsilon(\xi)>0.
    \label{ineq2000}
\end{equation}
Hence 
\begin{equation}
            \lim_{\xi \longrightarrow +\infty}\Dot{s}_f^\epsilon(\xi)=0^+.
            \label{lim4000}
        \end{equation}
 Using the limits \eqref{limA1}, \eqref{Alimsfesp0-} and \eqref{lim4000} we are able to compute:
\begin{equation}
    \lim_{\xi \longrightarrow +\infty} \textcolor{black}{d^\epsilon_{b^+}(\xi)}=0^+.
    \label{lim000}
\end{equation}
\textcolor{black}{Hence, there exists $\xi_{0,b^+} \in \mathbb{R}$ such that for all $\xi>\xi_{0,b^+}$: 
 \begin{equation}
     d^\epsilon_{b^+}(\xi)\geq 0
     \label{db+epsgrzero}
 \end{equation}}
 \textcolor{black}{Following the same logic as in \ref{wbminusvbminus}, assume by contradiction for $\epsilon>0$ sufficiently small: \begin{equation}
    \langle w_b^+, v_b^+\rangle=:a_b^+ < 0.
    \label{assumb+}
\end{equation}} Again, note that the scalar products do not vanish according to hypotheses $H_1$ and $H_6$., cf. Sandstede's remark in the statement of $H_6$ \textcolor{black}{\cite{Sandstede98} where this conclusion is discussed in more detail.
Then: 
\begin{equation}
    \forall\,\,\widetilde{\epsilon }> 0, \,\,\exists \,\,\ \xi_{1,b^+} \in \mathbb{R},\,\,\,\,\forall\,\,\xi > \xi_{1,b^+}: 
     \vert e^{\big(-2\lambda_2(X_2)+\frac{\mu_0(r)+\mu_b(r)}{D_0(r)}\big)\xi}d^\epsilon_{b^+}(\xi)-a_b^+ \vert < \widetilde{\epsilon }
     \label{ab+ineg}
\end{equation}
In particular, for $\widetilde{\epsilon }=\frac{\vert a_b^+ \vert}{2}>0$ and $\xi_{1,b^+}$ chosen accordingly such that \ref{ab+ineg} is satisfied, we have: 
\begin{equation}
    \forall\,\,\xi>\xi_{1,b^+}:e^{\big(-2\lambda_2(X_2)+\frac{\mu_0(r)+\mu_b(r)}{D_0(r)}\big)\xi}d^\epsilon_{b^+}(\xi) \in (a_b^+-\widetilde{\epsilon }, a_b^++\widetilde{\epsilon })=(\frac{3}{2}a_b^+, \frac{a_b^+}{2})
\end{equation}
which is a subset of $\mathbb{R}^-\backslash \{0\}$. Due to the positiveness of the exponential prefactor, we have in particular: 
\begin{equation}
    \forall\,\,\xi>\xi_{1,b^+}:d^\epsilon_{b^+}(\xi) <0
    \label{xi0b+}
\end{equation}
This time for $\xi>\max \{\xi_{0,b^+},\xi_{1,b^+}\}$ ($\xi_{0,b^+}$ as in \eqref{db+epsgrzero}) we have as well $d^\epsilon_{b^+}(\xi)<0$ as $d^\epsilon_{b^+}(\xi)\geq0$, which is a contradiction. Hence our assumption \eqref{assumb+} was false, and as a consequence for $\epsilon>0$ sufficiently small:
 \begin{equation}
     \langle w_b^+, v_b^+ \rangle>0.
 \end{equation}}
\item \underline{Computing $\langle w_f^-,v_f^-\rangle$}:\\ 

\textcolor{black}{Similarly to the work done in \ref{wbminusvbminus}, one shows for $\epsilon>0$ sufficiently small: $\langle w_f^-, v_f^- \rangle>0$. Corresponding details can be found in  Appendix \ref{Appendice w_f^-, v_f^-}.}

\item \underline{Computing $\langle w_f^+,v_f^+\rangle$}:\\ 

\textcolor{black}{Analogously to \ref{wbplusvbplus}, one shows for $\epsilon>0$ sufficiently small: $\langle w_f^+, v_f^+ \rangle>0$. Corresponding details can be found in Appendix \ref{Appendice w_f^+,v_f^+}.}

\end{enumerate}
Hence, the scalar products $\langle w_j^-,v_j^-\rangle$ and $\langle w_j^+,v_j^+\rangle$ are positive for $j\in\{f,b\}$, where $\epsilon>0$ is taken sufficiently small. 
\subsection{Existence theorem of \texorpdfstring{$N$}{Lg}-front and \texorpdfstring{$N$}{Lg}-back solutions}
\label{PS}
In this section, we define \textit{N-front} and \textit{N-back} solutions from an orbital perspective. Note that these definitions are equivalent to those given in \cite{EKR2022} for traveling waves in term of existence and number of pulses. \\

For $\epsilon>0$, we choose two \textcolor{black}{codimension-one hyperplanes (or ``sections'')} $\Sigma_j$ where $j\in\{f,b\}$ such that:
\begin{enumerate}
    \item $\Sigma_j$ \textcolor{black}{contains the point} $\gamma_j^\epsilon(0)$; 
    \item $\Sigma_j$ is transverse to the vector field.
\end{enumerate}
The heteroclinic solution $\gamma_f^\epsilon(\xi)$ is called \textit{simple front}, and the heteroclinic solution $\gamma_b^\epsilon(\xi)$ \textit{simple back}. A heteroclinic orbit which connects the equilibrium $X_1$ to the equilibrium $X_2$ and intersects the Poincaré section $\Sigma_b$ $N$ times, is called a \textit{N-front solution}. Analogously, a heteroclinic orbit connecting the steady state $X_2$ to the steady state $X_1$ and intersecting $\Sigma_f$ $N$ times is called a \textit{N-back solution}. Each $N$-front or $N$-back solution hence crosses the union of Poincaré sections  $\Sigma_f \cup \Sigma_b$ $2N+1$ times (\cite{Sandstede98}). 
\\ \\
Since Hypotheses $H_0$ to $H_6$ are fulfilled for the Barkley model \eqref{primary} under conditions 1. to 3. stated at the beginning of Section \ref{proofbark} (as shown in Sections \ref{Section2.1} to \ref{last section}), we can apply Theorem 1 in \cite{Sandstede98}, which gives existence of $N$-front and $N$-back traveling solutions and characterizes them. As in equation \eqref{abstract form}, we write the dynamical system \eqref{eqdiff} in the abstract form
\begin{equation}
    \Dot{X}=\mathcal{F}(X;\boldsymbol{\alpha}(r);r).
    \label{abstractform2}
\end{equation}
\textcolor{black}{Recall that $\alpha=(D,\mu,\epsilon)$ contains three parameters of the system with $\mu=-(\zeta+c)$, where $c$ is the wave speed. The next result is a variant of the existence of $N$-fronts and $N$-backs in the terminology of~\cite{Sandstede98}, where one shall view $\zeta$ and $D$ as parameters that are fixed depending upon the bifurcation parameters $\mu$ and $\epsilon$, while $r\in(\frac{2}{3},\beta)$ is in the allowed Reynolds number range. In \cite{EKR2022} another set of two parameters was used employing $r$ and $\epsilon$ as the bifurcation parameters, but evidently we always need precisely two free parameters for the unfolding.}

\begin{theorem}
\label{thm:Sandexist}
    For each $N>1$ \textcolor{black}{and fixed $r \in (\frac{2}{3},\beta),$} a unique curve \textcolor{black}{$\omega_N:[0,\rho_0)\rightarrow \mathbb{R}^2$ for $\omega_N=\omega_N(\rho)$} 
    exists in parameter space \textcolor{black}{$(\mu,\epsilon)\in\mathbb{R}^2$} and \eqref{abstractform2} has an $N$-front solution exactly for \textcolor{black}{$\omega=\omega_N(\rho)$} for some $\rho$. 
    The curve \textcolor{black}{$\omega_N$ is of class $C^1$ in $\rho$} and we have uniqueness of the $N$-fronts.
    \\Assume that \textcolor{black}{$\omega_1$} and \textcolor{black}{$\omega_2$} correspond to the existence of a simple front or back respectively. Then we can compute the so-called return times of the N-fronts with respect to the Poincaré sections $\Sigma_f$ and $\Sigma_b$, which are explicitly given by for $k\in \{0,1,...,N-1\}$:
    \begin{equation}
        \tau_i = \left\{
\begin{array}{ll}
\frac{\beta_2^\epsilon+\eta_{k+1}}{\lambda_2^\epsilon(X_1)}\ln(\rho) & i=2k+1, \, \textrm{corresponding to the time near the steady state $X_1$} \\
\frac{1}{\lambda_2^\epsilon(X_2)}\ln(\rho) & i=2k, \, \textrm{corresponding to the time near the steady state $X_2$} \\
\end{array},
\right. 
\label{Fallunt}
    \end{equation}
    where 
    \begin{equation}
        \beta_2^\epsilon:=\frac{\lambda_3^\epsilon(X_2)}{-\lambda_2^\epsilon(X_2)}>1
        \label{beta2eps}
    \end{equation}
    (cf. \textcolor{black}{Corollary \ref{Corcondspec} point 2}) and $(\eta_k)_{k\in \{0,1,...,N-1\}}$ being defined by following recurrence relation:
    \begin{equation}
        \eta_k=\beta_1^\epsilon \eta_{k+1}+\eta_{N-1}>\eta_{k+1}
    \end{equation}
    for $k\in \{0,1,...,N-2\}$ and $\eta_{N-1}=\beta_1^\epsilon \beta_2^\epsilon-1>0$. \\
    
    $\beta_1^\epsilon$ is defined analogously to $\beta_2^\epsilon$ as being the ratio:
    \begin{equation}
        \beta_1^\epsilon:=\frac{\lambda_3^\epsilon(X_1)}{-\lambda_2^\epsilon(X_1)}>1
    \end{equation}
    (cf. \textcolor{black}{Corollary \ref{Corcondspec} point 2}). \\
    
    We have similar results for $N$-back solutions. 
    \label{First Theorem}
\end{theorem}

For a given parameter configuration, we see from \eqref{Fallunt} that the time spent by the $N$-fronts near the equilibrium $X_2$ 
\begin{equation}
    \tau:= \frac{1}{\lambda_2^\epsilon(X_2)}\ln(\rho)
\end{equation} is identical for each layer $i=2k$. Computing the ratio of the time spent near $X_1$ for each layer $i = 2k+1$ over the time spent near $X_2$ yields:
\begin{equation}
    \frac{\tau_{2k+1}}{\tau}=\frac{-\lambda_3^\epsilon(X_2)+\eta_{k+1}\lambda_2^\epsilon(X_2)}{\lambda_2^\epsilon(X_1)}=:\sigma_{2k+1}.
\end{equation}
The sequence $(\sigma_{2k+1})_{k\in \{0,1,...,N-1\}}$ is strictly decreasing in $k$ as $(\eta_{k})_{k\in \{0,1,...,N-1\}}$ is also strictly decreasing and $\lambda_2^\epsilon(X_k)<0$ for $k\in\{1,2\}$. This means that the distance of the odd layers is getting smaller for increasing layer indices, as depicted in Figure \ref{fig: Layertime}. \textcolor{black}{Basically, Theorem~\ref{thm:Sandexist} is just a variant of the existence theorem of $N$-fronts and $N$-backs in~\cite{EKR2022} so we could have omitted it here in principle. Yet, it is re-assuring to see that the hard explicit calculations needed in~\cite{EKR2022} to show existence are precisely the calculations needed to apply the more abstract framework in~\cite{Sandstede98}. The aspect we aim to exploit here is that the abstract results in~\cite{Sandstede98} also contain a linear/spectral stability result for the waves, and this is what we want to utilize next. Yet, this requires checking additional hypotheses.}

\begin{figure}[H]
    \centering
    \includegraphics[width=16cm]{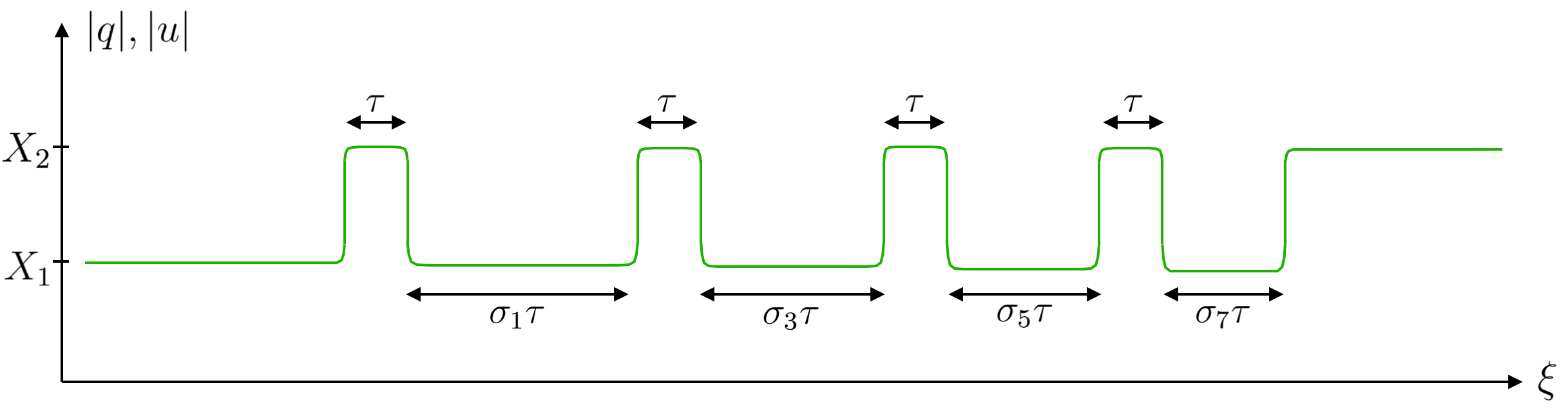}\\[3mm] 
    \caption{Representation of a 4-front wave solution. Note the constant distance for even layers and the decreasing distance for odd layers. The return times with respect to $\Sigma_f$ and $\Sigma_b$ are explicitly given in \eqref{Fallunt}.}
    \label{fig: Layertime}
    \end{figure}
\subsection{Hypothesis \texorpdfstring{$H_7$}{Lg}: Non-zeroness of Melnikov integrals}
\label{H7}
Now we want to describe the bounded solutions
\begin{equation}
    Y(\xi)=\begin{pmatrix}
        q(\xi)\\ s(\xi) \\ u(\xi)
    \end{pmatrix}\in C^1(\mathbb{R},\mathbb{R}^3)
\end{equation}
of the differential equation: 
\begin{equation}
   \Dot{Y}=\Big(D_X\mathcal{F}\big(\gamma_{f,N}(\rho)(\xi),\omega_N(\rho),r\big)+\lambda B(\xi)\Big)Y
   \label{eq251}
\end{equation}
for $\lambda \in \mathbb{C}$ belonging to a neighborhood of 0. The function $\gamma_{f,N}(\rho)$ denotes the $N$-front solution which exists for $\omega=\omega_N(\rho)$ (cf. Theorem \ref{First Theorem}) and $B$ a $3 \times 3$ real matrix-valued function. Equation \eqref{eq251} is nothing else than a generalized eigenvalue problem, which can be written as:
\begin{equation}
    LY=\lambda BY.
\end{equation}
Now if we compare the traveling wave system of equations \eqref{eqdiff} with equations \eqref{abstractform2} and \eqref{eq251}, they appear to be of the same form by taking $B$ as:
\begin{equation}
    B(\xi)=
\begin{pmatrix}
0 & 0 & 0 \\
\frac{1}{D} & 0 & 0 \\
0 & 0 & \frac{1}{c-u(\xi)}
\end{pmatrix}.
\end{equation}
Hypothesis $H_7$ supposes that the Melnikov integrals 
\begin{equation}
    M_j=\int_{- \infty}^{+ \infty}\langle \psi_{j}(\xi),B(\xi)\dot{\gamma}_j^\epsilon(\xi)\rangle d\xi
    \label{defmel}
\end{equation}
are non-zero for $j\in\{f,b\}$ with $\psi_{f}$ and $\psi_{b}$ as \eqref{expressionpsifb} (ie. chosen as in Hypothesis $H_6$). \textcolor{black}{Now we have to verify the hypothesis about the Melnikov integrals:} 
\begin{theorem}
    The Melnikov integrals $M_j$ as defined in 
\eqref{defmel} are negative for $j\in\{f,b\}$ and $\epsilon>0$ taken sufficiently small.
\label{Melnikovtheorem}
\end{theorem}
\begin{proof}
    First of all, note that for any solution ${\gamma}_j^\epsilon$ of the system of differential equations \eqref{eqdiff} we have:
    \begin{equation}
        B(\xi)\dot{\gamma}_j^\epsilon(\xi)=B(\xi)\begin{pmatrix}
\dot{q}_j^\epsilon(\xi) \\
\dot{s}_j^\epsilon(\xi)  \\
\dot{u}_j^\epsilon(\xi)  
\end{pmatrix}=\begin{pmatrix}
0 \\
\frac{\dot{q}_j^\epsilon(\xi)}{D}  \\
\frac{\dot{u}_j^\epsilon(\xi)}{c-u_j^\epsilon(\xi)}  
\end{pmatrix}=\begin{pmatrix}
0\\
\frac{s_j^\epsilon(\xi)}{D}  \\
-\frac{\epsilon g(q_j^\epsilon(\xi),u_j^\epsilon(\xi))}{(c-u_j^\epsilon(\xi))^2}
\end{pmatrix}=D_{\mu}F(q_j^\epsilon,s_j^\epsilon,u_j^\epsilon,\mu)
\label{dmu}
    \end{equation}
with $\mu=-\zeta-c$ as defined in Section \ref{Section 1} and 
\begin{equation}
    F(q_j^\epsilon,s_j^\epsilon,u_j^\epsilon,\mu):=\begin{pmatrix}
        s_j^\epsilon \\ \frac{1}{D}\big((u_j^\epsilon+\mu)s_j^\epsilon-f(q_j^\epsilon,u_j^\epsilon;r)\big) \\ \frac{\epsilon g(q_j^\epsilon, u_j^\epsilon)}{u_j^\epsilon-c}
    \end{pmatrix}
\end{equation}
the right-hand side of \eqref{eqdiff}. \textcolor{black}{The last step in equation \eqref{dmu} shows a key  relation how to express the matrix-vector multiplication $B(\xi)\dot{\gamma}_j^\epsilon(\xi)$ in terms of the derivative of the vector-valued function $F$. }Inserting equation \eqref{dmu} into expression \eqref{defmel} yields: 
\begin{equation}
    \int_{- \infty}^{+ \infty}\langle \psi_{j}(\xi),B(\xi)\dot{\gamma}_j^\epsilon(\xi)\rangle d\xi = \int_{- \infty}^{+ \infty}\langle \psi_{j}(\xi),D_{\mu}F(\gamma_j^\epsilon(\xi),\mu)\rangle d\xi.
    \label{obt1}
\end{equation}
The second integral in equation \eqref{obt1} is the derivative with respect to $\mu$ of the signed distance of the unstable and stable manifolds, which is measured in the direction $\psi_j(0)$, so we have:
\begin{equation}
\begin{split}
    \int_{- \infty}^{+ \infty}\langle \psi_{j}(\xi),D_{\mu}F(\gamma_j^\epsilon(\xi),\mu)\rangle d\xi=\frac{\partial}{\partial\mu}\langle \psi_{j}(0),X_\alpha^u-X_\alpha^s \rangle 
\end{split}
\label{obt2}
\end{equation}
with $X_\alpha^u$ being the unique intersection point between the unstable manifold $W^u_\alpha(X_{k_j})$ and the plane $\Sigma_j$ (where $j\in\{f,b\}$, $k_f=1$, $k_b=2$) which is defined as the plane perpendicular to the heteroclinic front or back $X_j(\xi)$ at $\xi=0$, for $\alpha$ close to $\alpha_0$ satisfying $X^u_{\alpha_0}=\gamma_j(0)$ 
and $X_\alpha^s$ defined analogously, \textcolor{black}{cf. \textcolor{black}{(R\ref{defQfQb})} and} \cite{Kok88}, \cite{lin_1990}, \cite{Den91b} \textcolor{black}{for a deeper analysis}. The Melnikov function $Q_j(\boldsymbol{\alpha};r): \mathcal{U}\times(\frac{2}{3},+\infty)\longrightarrow \mathbb{R}$, with $\mathcal{U} \subset \mathbb{R}^3$ a small neighborhood of $\alpha_0$, is smooth and measures the signed distance between the stable and unstable manifolds $W^s_\alpha(X_{k_{j+1}})$ and $W^u_\alpha(X_{k_j})$. \textcolor{black}{Using (R\ref{defQfQb}),} $Q_j$ satisfies:
\begin{equation}
    X_\alpha^s-X_\alpha^u=Q_j(\alpha;r)e_j.
\end{equation}
Hence, with the initial condition $\psi_j(0)=-e_j=(s_j'(0),-q_j'(0),0)$ (cf. equation \eqref{initcond}) we have:
\begin{equation}
    \frac{\partial}{\partial\mu}\langle \psi_{j}(0),X_\alpha^u-X_\alpha^s \rangle=\vert \vert e_j \vert \vert^2 \frac{\partial Q_j}{\partial\mu}(\alpha;r).
    \label{vertvert}
\end{equation}
\textcolor{black}{Building the bridge  between both expressions as above, is a key step enabling us to investigate the behaviour of $M_j$ by computing the derivative with respect to $\mu$ of the Melnikov integrals $Q_j$ evaluated at $\alpha=\alpha_0(r)$ using equation \ref{Mint}}. For the heteroclinic front, this yields with $s_f$ and $s_b$ as in expressions \eqref{gammafxi} and \eqref{la}: 
\begin{equation}
\begin{split}
    \frac{\partial Q_f}{\partial\mu}(\alpha_0(r);r)&=-\frac{1}{D_0(r)}\int_{-\infty}^{+\infty}e^{-\frac{\mu_0(r)+2}{D_0(r)}\xi}\begin{pmatrix}
        -s_f'(\xi;r) \\ s_f(\xi;r) \\ 0
    \end{pmatrix}
    \cdot 
    \begin{pmatrix}
        0 \\ s_f(\xi;r) \\ 0
    \end{pmatrix}d\xi \\&=-\frac{1}{D_0(r)}\int_{-\infty}^{+\infty}e^{-\frac{\mu_0(r)+2}{D_0(r)}\xi}s_f(\xi;r)^2 d\xi<0,
\end{split}
\end{equation}
which is negative, since $s_f$ is not identically equal to $0$ (we even have $s_f(\xi;r)>0$ for all $\xi \in \mathbb{R}$, \textcolor{black}{cf. Lemma \ref{Lemma derivatives} point 1.)}
Analogously, we obtain for the heteroclinic back: \\
   \begin{equation}
   \begin{split}
           \frac{\partial Q_b}{\partial\mu}(\alpha_0(r);r)&= -\frac{1}{D_0(r)}\int_{-\infty}^{+\infty}e^{-\frac{\mu_0(r)+u_b(r)}{D_0(r)}\xi}\begin{pmatrix}
        -s_b'(\xi;r) \\ s_b(\xi;r) \\ 0
    \end{pmatrix}
    \cdot 
    \begin{pmatrix}
        0 \\ s_b(\xi;r) \\ 0
    \end{pmatrix}d\xi \\&=-\frac{1}{D_0(r)}\int_{-\infty}^{+\infty}e^{-\frac{\mu_0(r)+u_b(r)}{D_0(r)}\xi}s_b(\xi;r)^2 d\xi<0,
    \end{split}
    \end{equation}
    which is also negative, since $s_b$ is not identically equal to $0$ (we even have $s_b(\xi;r)<0$ for all $\xi \in \mathbb{R}$ \textcolor{black}{cf. Lemma \ref{Lemma derivatives} point 3.}
Hence, 
\begin{equation}
    \frac{\partial Q_j}{\partial\mu}(\alpha_0(r);r)<0
    \label{Qjalpha0}
\end{equation}
    for $j\in\{f,b\}$. We know that the Melnikov functions $Q_j(\cdot;r)$ are smooth \textcolor{black}{(cf. (R\ref{defQfQb}))}, hence we can extend result \eqref{Qjalpha0} to parameter values $\alpha=(D,\mu,\epsilon)$ belonging to a small neighborhood of $\alpha_0$ yielding
\begin{equation}
    \frac{\partial Q_j}{\partial\mu}(\alpha(r);r)<0
    \label{Qjalpha}
\end{equation}
for $\epsilon>0$ sufficiently small and $j\in\{f,b\}$. Combinining inequality \eqref{Qjalpha} with the chain of equalities \eqref{defmel}, \eqref{obt1}, \eqref{obt2} and \eqref{vertvert}, we conclude for the Melnikov integrals:
\begin{equation}
    M_j<0
\end{equation}
for $\epsilon>0$ sufficiently small.
\end{proof}
In particular $M_j$ with $j\in\{f,b\}$ are non-zero for $\epsilon>0$ taken sufficiently small, which shows that Hypothesis $H_7$ is satisfied. 
\subsection{Stability theorem for the \texorpdfstring{$N$}{Lg}-front and \texorpdfstring{$N$-back solutions}{Lg}}
As shown in Sections \ref{Section2.1} to \ref{last section} and \ref{H7}, Hypotheses $H_0$ to $H_7$ are satisfied for the Barkley model \eqref{primary} under conditions 1. to 3. stated at the beginning of Section \ref{proofbark}. Hence we can apply Theorem 2 in \cite{Sandstede98}, which establishes stability of the $N$-front and $N$-back solutions. 
\begin{theorem}
There exists $\delta>0$, which does not depend on $N$, such that, for each $N>1$ and $\rho_0>0$ sufficiently small, equation \eqref{eq251} has exactly $2N+1$ solutions $(\lambda_i,Y_i) \in \mathbb{C} \times C^1(\mathbb{R},\mathbb{R}^3)$ with $\vert \lambda \vert < \delta$, hence $\lambda \in U_\delta(0)$. The eigenvalues are counted with multiplicity and satisfy:
\begin{equation}
        \lambda_i = \left\{
\begin{array}{ll}
(a_{2k-1}+o(1))\rho & \, \textrm{for \,\,$i=2k-1$} \\
(a_{2k}+o(1))\rho^{{\beta_2^\epsilon}+\eta_k}  & \, \textrm{for \,\,$i=2k$} \\
0  & \, \textrm{for \,\,$i=2N+1$}
\end{array}
\right.
    \end{equation}
    as $\rho \longrightarrow 0^+$, where $k\in\{1,2,...,N\}$ and $\eta_k$ as in Theorem \ref{First Theorem}. $\beta_2^\epsilon$ is defined in \eqref{beta2eps}. $a_{2k-1}$ and $a_{2k}$ are non-vanishing constants, which sign is related to the sign of the Melnikov integrals $M_j$ from \eqref{defmel} as follows:
    \begin{equation}
         \text{sign}(a_i) = \left\{
\begin{array}{ll}
\text{sign} (M_f) & \, \textrm{for \,\,$i=2k$} \\
\text{sign} (M_b) & \, \textrm{for \,\,$i=2k-1$}
\end{array}.
\right.
    \end{equation}
    From Theorem \ref{Melnikovtheorem} we know that $M_f<0$ and $M_b<0$ for $\epsilon>0$ sufficiently small, hence the eigenvalues $\lambda_i$ are all located in the left half plane. These results hold for $N$-fronts and $N$-backs.
\end{theorem}

\textcolor{black}{The last result shows the spectral (or linear) stability of the $N$-fronts and $N$-backs. If one wanted to also prove (local) asymptotic orbital stability, one would also have to check hypotheses to ensure that the linearized spectral stability extends~\cite{KapitulaPromislow} to the nonlinear setting. If the linearization of the problem gives a semigroup generated by a sectorial operator, then one gets nonlinear stability automatically from linear stability~\cite{Henry}. In the context of the Barkley model, the $u$-component describing the centerline velocity is the problematic component for nonlinear stability. Indeed, for the standard FitzHugh-Nagumo without diffusion in the gating variable component, it is known (cf. e.g.~\cite{ArioliKoch}) that also nonlinear stability holds. However, the Barkley model we studied here contains a leading derivative term $-uu_x$ making this previous analysis not directly applicable. Hence, we pose this as an interesting future problem to study nonlinear stability in more detail.}

\section{Conclusion}
Starting from the Barkley model \eqref{primary} for pipe flow \textcolor{black}{which was extensively verified in experiments}, we obtained the system of ODEs \eqref{eqdiff} as we were investigating traveling wave solutions. Such waves exhibited by the model show different profile types, $N$-front solutions consist of the concatenation of $2N+1$ copies of a simple front and back. In terms of orbits, this means that a $N$-front solution crosses the Poincaré section as defined in Section \ref{PS} $N$ times. After having verified the existence of two hyperbolic steady states and spectral properties of the linearized vector field at these points, we used that two heteroclinic orbits exist between them. These results hold for all $r>\frac{2}{3}$. For wider discussion and explicit computation of the equilibria dependency on the model Reynolds number $r$ please refer to \cite{Master_s_Thesis_Pascal_Sedlmeier} Section 4. The non-degeneracy condition of the heteroclinic loop forces us to restrict the applicable domain, since only for $r\in(\frac{2}{3},\beta)$ with $\beta\approx 0.72946$ the Melnikov integral $\widetilde{M_f}(r)$ as defined in Section \ref{H3} is positive. After checking the linear independence of the Melnikov integrals $\nabla Q_f$ and $\nabla Q_b$, the strong inclination property and the positivity of the scalar products ${\langle w_j^\pm,v_j^\pm \rangle}_{j \in \{f,b\}}$ for $\epsilon>0$ sufficiently small, we were able to conclude the existence of $N$-front and $N$-back solutions for each $N>1$. The existence analysis based upon abstract conditions from~\cite{Sandstede98} is basically a re-interpretation, augmented by some additional technical steps, based upon the extensive analysis via geometric singular perturbation theory and Melnikov calculations in~\cite{EKR2022}. In addition, it is a key observation of this work that using a more abstract framework allowed us to also directly check just one more hypothesis from~\cite{Sandstede98}, Hypothesis $H_7$, regarding the Melnikov integrals ${M_j}$ for $j \in \{f,b\}$ as defined in Section \ref{H7}. Although checking $H_7$ is non-trivial, it can be verified to conclude directly together with Hypotheses $H_0$ to $H_6$  the (local linearized asymptotic) stability for the $N$-front and $N$-back solutions. This conclusion is crucial because it explains, why quite a wide variety of puff/slug patterns are observed in experiments in the transition to turbulence regime.\\

\akn{CK and PS would like to thank Paul Carter, Maximilian Engel and Björn de Rijk for fruitful discussions and valuable interaction. CK would like to thank the VolkswagenStiftung for support via a Lichtenberg Professorship. CK also would like to thank the DFG for support via a Sachbeihilfe grant project number 444753754. \textcolor{black}{We also thank a referee for the very helpful comments that led to many improvements in the manuscript.}}

\begin{appendices}
\section{\textcolor{black}{Proof of \texorpdfstring{$\langle w_f^-, v_f^- \rangle>0$}{Lg} for \texorpdfstring{$\epsilon>0$}{Lg} sufficiently small}}
\label{Appendice w_f^-, v_f^-}

Analogously to the work done in Section \ref{last section} \ref{wbminusvbminus}, we first show the result for $\epsilon=0$ (corresponding to the fast subsystem \eqref{fastss}) and then extend it to $\epsilon>0$ sufficiently small. First of all, for  $\epsilon=0$ we have
\textcolor{black}{\begin{equation}    
    \begin{split}
         \langle w_f^-, v_f^- \rangle &=\langle \lim_{\xi \longrightarrow +\infty}e^{\lambda_3(X_1)\xi}\psi_b(\xi) , \lim_{\xi \longrightarrow +\infty}e^{\lambda_3(X_1)\xi}\dot{X}_f(-\xi) \rangle.
    \end{split}
\end{equation}}
We use the linearity and the continuity of the scalar product, which allows us to write:
\begin{equation}
     \begin{split}
     \langle w_f^-, v_f^- \rangle &= \lim_{\xi \longrightarrow +\infty}e^{2\lambda_3(X_1)\xi} \langle \psi_b(\xi) , \dot{X}_f(-\xi) \rangle.
     \end{split}
\end{equation}
Using the definition of the heteroclinic connection $X_f$ and expression \eqref{expressionpsifb} for the solution $\psi_b$ of the adjoint variational equation \eqref{adjvareq} with initial condition \eqref{initcond}, we obtain:
     \begin{equation}
     \begin{split}
     \langle w_f^-, v_f^- \rangle 
      &=\lim_{\xi \longrightarrow +\infty}e^{\big(2\lambda_3(X_1)-\frac{\mu_0(r)+u_b(r)}{D_0(r)}\big)\xi} \big(\dot{s}_b(\xi) s_f(-\xi) - s_b(\xi) \dot{s}_f(-\xi) \big)
     \end{split}
     \label{li}
     \end{equation}
with $s_f$, $s_b$ and $u_f$ as defined in equations \eqref{gammafxi} and \eqref{la}.\textcolor{black}{We define \begin{equation}
    d^0_{f^-}(\xi):=\dot{s}_b(\xi) s_f(-\xi) - s_b(\xi) \dot{s}_f(-\xi)
\end{equation}
to be the second factor in above equation. By contradiction, assume that \begin{equation}
    \langle w_f^-, v_f^- \rangle =:a_f^- < 0.
    \label{assumf-}
\end{equation} Again, as for point \ref{wbminusvbminus}, the scalar products do not vanish according to hypotheses $H_1$ and $H_6$. Following the reasoning in \ref{wbminusvbminus}, we show analogously, as the exponential prefactor in \ref{li} is positive, the existence of a $\xi_{0,f^-} \in \mathbb{R}$, such that:  \begin{equation}
     \forall\,\,\xi>\xi_{0,f^-}: d^0_{f^-}(\xi) < 0.
    \label{xi0f-}
\end{equation}} 
\textcolor{black}{Expressions 1. to 4. in Lemma \ref{Lemma derivatives}} allow us to explicitly compute the following limits as $\xi \longrightarrow +\infty$:
\begin{equation}
    \lim_{\xi \longrightarrow +\infty}s_f(-\xi)=0^+,\,\,\,\,\,\,\,\lim_{\xi \longrightarrow +\infty}\Dot{s}_f(-\xi)=0^+,\,\,\,\,\,\,\,\lim_{\xi \longrightarrow +\infty}s_b(\xi)=0^-,\,\,\,\,\,\,\,\lim_{\xi \longrightarrow +\infty}\Dot{s}_b(\xi)=0^+.
    \label{limits194}
\end{equation}

Hence we conclude:
\begin{equation}
 \lim_{\xi \longrightarrow +\infty} \textcolor{black}{d^0_{f^-}(\xi)}=0^+.
 \label{lim100}
\end{equation}

\textcolor{black}{Consequently, there exists $\xi_{1,f^-} \in \mathbb{R}$ such that for all $\xi>\xi_{1,f^-}$: 
 \begin{equation}
     d^0_{f^-}(\xi)\geq0
 \end{equation}
 Now for $\xi>\max \{\xi_{0,f^-},\xi_{1,f^-}\}$ we have as well $d^0_{f^-}(\xi)<0$ as $d^0_{f^-}(\xi)\geq0$, which is a contradiction. Hence, our assumption \eqref{assumf-} was false and as a result for $\epsilon=0$:
 \begin{equation}
     \langle w_f^-, v_f^- \rangle>0.
\label{resultwf-}
 \end{equation}}
Following the same logic as in point \ref{wbminusvbminus}, we generalize  result \eqref{resultwf-} to $\epsilon>0$ sufficiently small. Hence, let $\gamma^\epsilon_f$ be a heteroclinic front solution of the Barkley model \eqref{eqdiff} with parameter $\epsilon$ taken sufficiently small, connecting the equilibria $X_1$ to $X_2$. We define the components of $\gamma^\epsilon_f$ in the $(q,s,u)$-frame as:
\begin{equation}
    \gamma^\epsilon_f(\xi)=:
\begin{pmatrix}
q_{f}^\epsilon(\xi)  \\
s_{f}^\epsilon(\xi) \\
u_{f}^\epsilon(\xi)
\end{pmatrix}.
\end{equation}
We need to determine, similarly to the reasoning for $\epsilon=0$, the limit of
\begin{equation}
 \textcolor{black}{
d^\epsilon_{f^-}(\xi):=\dot{s}_b(\xi)s_f^\epsilon(-\xi)-s_b(\xi)\dot{s}_f^\epsilon(-\xi)}
\end{equation}
\textcolor{black}{as $\xi$ tends to $+\infty$.}
The exponential prefactor
\begin{equation}
    e^{\big(2\lambda_3(X_1)-\frac{\mu_0(r)+u_b(r)}{D_0(r)}\big)\xi}
\end{equation}
remains unchanged (cf. equation \eqref{li}). 
Using \textcolor{black}{Lemma \ref{Lemma derivatives} point 1.} we obtain for the $s$-component of the singular heteroclinic front:
\begin{equation}
    \forall\,\,\xi \in \mathbb{R}: s_f(\xi)>0.
\end{equation}
For $\epsilon>0$ sufficiently small, $\gamma_f^\epsilon$ lies in the vicinity of the singular heteroclinic front, which implies:
\begin{equation}
    \exists\,\,\xi_0 \in \mathbb{R}\,\,\,\, \forall\,\,\xi < \xi_0: s_f^\epsilon(\xi)>0.
    \label{ineq20}
\end{equation}
And since 
\begin{equation}
    \lim_{\xi \longrightarrow -\infty}\gamma_f^\epsilon(\xi)=X_1=(0,0,2),
    \label{hequ1}
\end{equation}
we immediately have
\begin{equation}
    \lim_{\xi \longrightarrow +\infty}s_f^\epsilon(-\xi)=0.
\end{equation}
Together with inequality 
\eqref{ineq20} this means:
\begin{equation}
    \lim_{\xi \longrightarrow +\infty}s_f^\epsilon(-\xi)=0^+
    \label{sflimeps}
\end{equation}
for $\epsilon>0$ sufficiently small. \\ 
\newline 
The derivative of $s_f^\epsilon (-\xi)$ satisfies the differential equation (cf. the dynamical system \eqref{eqdiff}):
\begin{equation}
    \Dot{s}_f^\epsilon=\frac{1}{D}\Big((\underbrace{u_f^\epsilon+\mu}_{A})\underbrace{s_f^\epsilon}_{B}-\underbrace{q_f^\epsilon}_{C}\big(\underbrace{r+u_f^\epsilon-2-(r+0.1)(q_f^\epsilon-1)^2}_{D}\big)\Big).
    \label{ABCD2}
\end{equation}
We investigate the behaviour of $\Dot{s}_f^\epsilon(-\xi)$ as $\xi \longrightarrow + \infty$ by computing the following limits for the different terms from equation \eqref{ABCD2}:\\
\par
\begingroup
\leftskip=1cm 
\noindent 
\underline{Term $A$:}
\begin{equation}
\lim_{\xi \longrightarrow +\infty}u_f^\epsilon(-\xi;r)+\mu=2-\zeta-c>0
\end{equation}
for $\zeta>0$ sufficiently small, since $c<\frac{4}{3}$ for $r>\frac{2}{3}$ (cf. beginning of Section \ref{proofbark}, hypotheses). \\
\par
\endgroup
\par
\begingroup
\leftskip=1cm 
\noindent 
\underline{Term $B$:}
\begin{equation}
    \lim_{\xi \longrightarrow +\infty}s_f^\epsilon(-\xi)
\end{equation} is given by equation \eqref{sflimeps}.\\
\par
\endgroup
\par
\begingroup
\leftskip=1cm 
\noindent 
\underline{Term $C$:} 
\begin{equation}
\lim_{\xi \longrightarrow +\infty}q_f^\epsilon(-\xi;r)=0^+,
\end{equation}
since 
\begin{equation}
\lim_{\xi \longrightarrow +\infty}q_f^\epsilon(-\xi;r)=0
\end{equation}
(cf. equation \eqref{hequ1}) and
\begin{equation}
    \exists\,\,\xi_0 \in \mathbb{R}\,\,\,\,\forall\,\,\xi < \xi_0: q_f^\epsilon(\xi;r)>0.
    \label{ineq30}
\end{equation}
Inequality \eqref{ineq30} follows from
\begin{equation}
      q_f(\xi;r)>0
    \label{ineq40}
\end{equation}
for all $\xi \in \mathbb{R}$ \textcolor{black}{(cf. (R(\ref{S3.5.1}))} and the fact that $\gamma_f^\epsilon$ lies in the vicinity of the singular heteroclinic front for $\epsilon>0$ sufficiently small. \\
\par
\endgroup
\par
\begingroup
\leftskip=1cm 
\noindent 
\underline{Term $D$:} The parabola branch of the manifold $M_0$ intersects the plane $q=0=q^*$ at $u=u^*$ such that:
\begin{equation}
    r+u^*-2-(r+0.1)(q^*-1)^2=0,
\end{equation}
which yields
\begin{equation}  
u^*=2.1.
\end{equation}
Since
\begin{equation}
    \lim_{\xi \longrightarrow +\infty}\gamma_f^\epsilon(-\xi;r)=X_1=(0,0,2),
\end{equation}
the $u$-component of $X_1$ is smaller than $u^*$:
\begin{equation}
    2<u^*,
\end{equation}
which means that $\gamma_f^\epsilon$ is located "underneath" the parabola of the manifold $M_0$ in the $(q,u)$-plane in a neighborhood of $-\infty$ (cf. Figure \ref{fig: Singloop}). Hence 
\begin{equation}
    u_f^\epsilon<2-r+(r+0.1)(q_f^\epsilon-1)^2
\end{equation}
meaning that the quantity 
\begin{equation}
r+u_f^\epsilon-2-(r+0.1)(q_f^\epsilon-1)^2<0
\end{equation} is negative in such a neighborhood.\\

Hence we have with expression \eqref{ABCD2} for $\Dot{s}_f^\epsilon$:
\begin{equation}
    \lim_{\xi \longrightarrow +\infty}\Dot{s}_f^\epsilon(-\xi)=0^+
    \label{eq263}
\end{equation}
for $\epsilon>0$ sufficiently small. \\
\par
\endgroup
With expressions \eqref{limits194}, \eqref{sflimeps} and \eqref{eq263} we are now able to compute the limit:
\begin{equation}
    \lim_{\xi \longrightarrow +\infty} \textcolor{black}{d^\epsilon_{f^-}(\xi)}=0^+.
\end{equation}
Analogously to the first case $\epsilon=0$, we then conclude for $\epsilon>0$ sufficiently small:
\begin{equation}
\langle w_f^-, v_f^- \rangle>0.
\end{equation}

\section{\textcolor{black}{Proof of \texorpdfstring{$\langle w_f^+,v_f^+ \rangle>0$}{Lg} for \texorpdfstring{$\epsilon>0$}{Lg} sufficiently small}}
\label{Appendice w_f^+,v_f^+}
Using the definitions \eqref{firstlimit} and \eqref{defwk} of $v_f^+$ and $w_f^+$, we can write for the scalar product  
     \textcolor{black}{\begin{equation}
     \begin{split}
         \langle w_f^+, v_f^+ \rangle  &=\langle \lim_{\xi \longrightarrow +\infty}e^{-\lambda_2(X_1)\xi}\psi_f(-\xi) , \lim_{\xi \longrightarrow +\infty}e^{-\lambda_2(X_1)\xi}\dot{\gamma}_b^\epsilon(\xi) \rangle.
    \end{split}
     \end{equation}}
     With the linearity and the continuity of the scalar product, we obtain:
      \begin{equation}
     \begin{split}
     \langle w_f^+, v_f^+ \rangle  &=\lim_{\xi \longrightarrow +\infty}e^{-2\lambda_2(X_1)\xi} \langle \psi_f(-\xi) , \dot{\gamma}_b^\epsilon(\xi) \rangle.
     \end{split}
     \end{equation}
     Using the definition of the orbit $\gamma_b^\epsilon$ and expression \eqref{expressionpsifb} for the solution $\psi_f$ of the adjoint variational equation \eqref{adjvareq} with initial condition \eqref{initcond}, we obtain
    \begin{equation}
     \begin{split}
     \langle w_f^+, v_f^+ \rangle
      &=\lim_{\xi \longrightarrow +\infty}e^{\big(-2\lambda_2(X_1)+\frac{\mu_0(r)+2}{D_0(r)}\big)\xi} \big( \dot{s}_f(-\xi) s_b^\epsilon(\xi)-s_f(-\xi) \dot{s}_b^\epsilon(\xi) \big).
     \end{split}
     \label{exp10}
     \end{equation}
    \textcolor{black}{For simplicity of notation, we define: \begin{equation}
        d^\epsilon_{f^+}(\xi):=\dot{s}_f(-\xi) s_b^\epsilon(\xi)-s_f(-\xi) \dot{s}_b^\epsilon(\xi)
    \end{equation}} 
     Now we investigate the behaviour of $s_b^\epsilon(\xi)$ and $\Dot{s}_b^\epsilon(\xi)$ as $\xi \longrightarrow + \infty$ of the actual heteroclinic back for $\epsilon>0$ sufficiently small. Similarly to the work done in \ref{wbplusvbplus} for the heteroclinic front, we have that $W^s_{\boldsymbol{\alpha}(\epsilon;r)}(K_{1,\boldsymbol{\alpha}(\epsilon;r)})$ and $W^u_{\boldsymbol{\alpha}(\epsilon;r)}(K_{2,\boldsymbol{\alpha}(\epsilon;r)})$ intersect transversally along $\gamma_b^\epsilon$(cf. \textcolor{black}{(R\ref{L3.7})}). As the heteroclinic loop is non-degenerate, $\gamma^\epsilon_b(\xi)$ is asymptotically tangent for $\xi \longrightarrow +\infty$ to the principal stable eigenvector $e_2(X_1)$ of the steady state $X_1=(0,0,2)$, cf. Section \ref{H3}. The (two-dimensional) stable manifold $W_{\boldsymbol{\alpha}}^s(K_{1,\boldsymbol{\alpha}})$ of the invariant manifold $K_{1,\boldsymbol{\alpha}}$ consists of all orbits (locally) converging to $K_{1,\boldsymbol{\alpha}}$ as $\xi \longrightarrow + \infty$, whereas the (one-dimensional) unstable manifold $W_{\boldsymbol{\alpha}}^u(K_{2,\boldsymbol{\alpha}})$ of the invariant manifold $K_{2,\boldsymbol{\alpha}}$ of all orbits (locally) converging to $K_{2,\boldsymbol{\alpha}}$ as $\xi \longrightarrow - \infty$.  \textit{Geometric singular perturbation theory} yields that  $W_{\boldsymbol{\alpha}}^s(K_{1,\boldsymbol{\alpha}})$ coincides for $\epsilon>0$ with the stable manifold $W_{\boldsymbol{\alpha}}^s(X_1)$ of the steady state $X_1$ in system \eqref{eqdiff} (cf. \textcolor{black}{(R\ref{defQfQb})}). The unstable manifold $W^u_{\boldsymbol{\alpha}}(K_{2,\boldsymbol{\alpha}})$ of the invariant manifold $K_{2,\boldsymbol{\alpha}}$ and the stable manifold $W^s_{\boldsymbol{\alpha}}(K_{1,\boldsymbol{\alpha}})$ of the invariant manifold $K_{1,\boldsymbol{\alpha}}$ depend smoothly on $\boldsymbol{\alpha}$ for $\boldsymbol{\alpha}$ close to $\boldsymbol{\alpha_0}$. This prevents any oscillatory behaviour of the actual heteroclinic back for $\epsilon>0$ along the singular one, which arises at $\boldsymbol{\alpha}=\boldsymbol{\alpha}_0$.\\ 

The orbit of $\gamma_b^\epsilon$ approches the singular heteroclinic back as $\epsilon \longrightarrow 0^+$, which is made up of the heteroclinic $X_b$ in the fast subsystem \eqref{fastss} in the layer $u=2$ and the slow orbit segment in the slow subsystem \eqref{slowsubsys}, which connects $Y_2$ to $X_1$. Since 
\begin{equation}
    \lim_{\xi \longrightarrow + \infty}X_b(\xi)=Y_2=(0,0,u_b(r)),
\end{equation}
we have:
\begin{equation}
    \lim_{\xi \longrightarrow + \infty}s_b(\xi)=0
\end{equation} 
(cf. also \eqref{limits194}), hence for $\epsilon>0$ sufficiently small, $s_b^\epsilon$ will be arbitrarily close to 0 for $X_b$ in the vicinity of $Y_2$. Using now inequality \eqref{sbinf}, we know that from $X_2$ to $X_1$, $s_b^\epsilon$ is firstly negative for $\epsilon>0$ sufficiently small and hence positive after the zero-crossing in the vicinity of $Y_2$, as $\gamma_b^\epsilon$ cannot show an oscillatory behaviour. Hence 
\begin{equation}
    \exists\,\,\xi_0 \in \mathbb{R}\,\,\,\,\forall\,\,\xi > \xi_0: s_b^\epsilon(\xi)>0
    \label{ineq123}.
\end{equation}
From 
\begin{equation}
    \lim_{\xi \longrightarrow + \infty}\gamma_b^\epsilon(\xi)=X_1=(0,0,2)
\end{equation} 
we have immediately that
\begin{equation}
    \lim_{\xi \longrightarrow + \infty}s_b^\epsilon(\xi)=0,
    \label{limimexico}
\end{equation} which means together with inequality \eqref{ineq123}:
\begin{equation}
    \lim_{\xi \longrightarrow + \infty}s_b^\epsilon(\xi)=0^+.
    \label{limimexico+}
\end{equation} 
Now we compute the limit of the derivative $\Dot{s}_b^\epsilon(\xi)$ as $\xi \longrightarrow +\infty$  for $\epsilon>0$ sufficiently small. From the dynamical system \eqref{eqdiff}, we know that the derivative satisfies
        \begin{equation}
    \Dot{s}_b^\epsilon(\xi)=\frac{1}{D}\Big((u_b^\epsilon(\xi)+\mu)s_b^\epsilon(\xi)-f(q_b^\epsilon(\xi),u_b^\epsilon(\xi);r)\Big).
        \end{equation}
        Using the continuity of the function $f$, we can write:
        \begin{equation}
    \lim_{\xi \longrightarrow +\infty}f(q_b^\epsilon(\xi),u_b^\epsilon(\xi);r)=f(0,2;r)=0,
    \label{f00}
        \end{equation}
        since $X_1$ lies on the critical manifold $M_0$. As a result, the limits \eqref{limimexico} and \eqref{f00} yield:
        \begin{equation}
            \lim_{\xi \longrightarrow +\infty}\Dot{s}_b^\epsilon(\xi)=0.
        \end{equation}
As equation \eqref{limimexico+} holds and $s_b^\epsilon$ cannot show an oscillatory behaviour, we have
\begin{equation}
    \exists\,\,\xi_0 \in \mathbb{R}\,\,\,\,\forall\,\,\xi > \xi_0: \Dot{s}_b^\epsilon(\xi)<0.
    \label{ineq20000}
\end{equation}
Hence 
\begin{equation}
            \lim_{\xi \longrightarrow +\infty}\Dot{s}_b^\epsilon(\xi)=0^-.
            \label{lim40000}
        \end{equation}
 From the explicit expressions for $s_f$ and $\Dot{s}_f$ (cf. \textcolor{black}{Lemma \ref{Lemma derivatives}}), we compute the following limits as $\xi \longrightarrow -\infty$:
\begin{equation}
    \lim_{\xi \longrightarrow -\infty} s_f(\xi)=0^+,\,\,\,\lim_{\xi \longrightarrow -\infty} \Dot{s}_f(\xi)=0^+.
\label{limu}
\end{equation}

 With the limits \eqref{limimexico+}, \eqref{lim40000} and \eqref{limu} we obtain:
\textcolor{black}{\begin{equation}
    \lim_{\xi \longrightarrow +\infty}  d^\epsilon_{f^+}(\xi)=0^+.
    \label{lim0000}
\end{equation}
Hence, there exists $\xi_{0,f^+} \in \mathbb{R}$ such that for all $\xi>\xi_{0,f^+}$: 
 \begin{equation}
     d^\epsilon_{f^+}(\xi)\geq 0
 \end{equation}}

\textcolor{black}{By contradiction, assume that \begin{equation}
        \langle w_f^+, v_f^+ \rangle=:a_f^+ < 0
        \label{assumf+}
    \end{equation} Again, the scalar products do not vanish according to hypotheses $H_1$ and $H_6$. Similarly to the computation done in Section \ref{last section} \ref{wbplusvbplus}, there exists $\xi_{1,f^+} \in \mathbb{R}$ such that: \begin{equation}
    \forall\,\,\xi>\xi_{1,f^+}: d^\epsilon_{f^+}(\xi) < 0.
    \label{xi0f+}
\end{equation}}
\textcolor{black}{
 Now for $\xi>\max \{\xi_{0,f^+},\xi_{1,f^+}\}$ we have as well $d^\epsilon_{f^+}(\xi)<0$ as $d^\epsilon_{f^+}(\xi)\geq 0$, which is a contradiction. Hence, our assumption \eqref{assumf+} was false and as a consequence for $\epsilon>0$ sufficiently small:
 \begin{equation}
     \langle w_f^+, v_f^+ \rangle>0.
\label{resultwf+}
 \end{equation}} 
\end{appendices}

\end{document}